%% file: amendedmetriccurrentpaper.tex
\documentclass[10pt]{amsart}
\usepackage{verbatim}
\usepackage{graphicx}

\def\BBR {{\mathbb R}}
\def\BBG {{\mathbb G}}
\def\BBH {{\mathbb H}}
\def\BBV {{\mathbb V}}
\def\BBE {{\mathbb E}}

\newcommand{\cdist}{\dist_{\operatorname{CC}}}
\newcommand{\rdist}{\dist_{\operatorname{R}}}

\newcommand{\gtimes}{\cdot}
\newcommand{\gunit}{e}
\newcommand{\restr}[2][\big]{\kern -.1em #1|_{#2}}

\newcommand{\charfcn}{\chi}

\newcommand{\contc}{C_c}

\newcommand{\lipc}{\operatorname{Lip_c}}
\newcommand{\lip}{\operatorname{Lip}}
\newcommand{\liploc}{\operatorname{Lip_{loc}}}

\newcommand{\bounded}{\operatorname{\mathcal B^\infty}}
\newcommand{\boundedc}{\operatorname{\mathcal B_c^\infty}}
\newcommand{\boundedloc}{\operatorname{\mathcal B_{loc}^\infty}}

\newcommand{\currents}[1]{\operatorname{\mathbf{M}}_{#1}}
\newcommand{\currentsloc}[1]{\operatorname{\mathbf{M}}_{#1}^{\operatorname{loc}}}
\newcommand{\currentslocabs}[1]{\operatorname{\mathbf{M}}_{#1,\operatorname{abs}}^{\operatorname{loc}}}
\newcommand{\currentslocabsvanishing}[1]{\operatorname{\mathbf{M}}_{#1,\operatorname{abs}}^{\operatorname{loc},H}}
\newcommand{\normal}[1]{\operatorname{\mathbf{N}}_{#1}}
\newcommand{\normalloc}[1]{\operatorname{\mathbf{N}}_{#1}^{\operatorname{loc}}}
\newcommand{\normallocvanishing}[1]{\operatorname{\mathbf{N}}_{#1}^{\operatorname{loc},H}}

\newcommand{\heiscurrents}[1]{\mathcal D_{\BBH, k}}
\newcommand{\lcurrents}[1]{\operatorname{\mathcal{D}}_{#1}}

\newcommand{\forms}[1]{{\mathcal{D}}^{#1}}
\newcommand{\formsc}[1]{{\mathcal{D}}^{#1}_{\operatorname{c}}}
\newcommand{\extforms}[1]{{\mathcal{\tilde{D}}^{#1}}}
\newcommand{\extformsc}[1]{{\mathcal{\tilde{D}}}^{#1}_{\operatorname{c}}}

\newcommand{\smoothforms}[1]{\operatorname{\mathcal{S} }^{#1}}
\newcommand{\smoothformsc}[1]{\operatorname{\mathcal{S}}^{#1}_{\operatorname{c}}}
\newcommand{\smoothextforms}[1]{\operatorname{\mathcal{\tilde{S}}}^{#1}}

\newcommand{\smooth}{C^\infty}
\newcommand{\smoothc}{C^\infty_{\operatorname{c}}}

\newcommand{\measforms}[1]{{\mathcal{E}}^{#1}}
\newcommand{\measformsc}[1]{{\mathcal{E}}^{#1}_{\operatorname{c}}}
\newcommand{\measextforms}[1]{\mathcal{\tilde{E}}^{#1}}
\newcommand{\measextformsc}[1]{{\mathcal{\tilde{E}}}^{#1}_{\operatorname{c}}}

\newcommand{\diffforms}[1]{\operatorname{\Omega}^{#1}}

\newcommand{\Dpansu}{D_c}
\newcommand{\dpansu}{d_c}
\newcommand{\Dpansuh}{D_c^h}
\newcommand{\dpansuh}{d_c^h}
\newcommand{\driem}{d_r}
\newcommand{\diff}[1]{d^{#1}}

\newcommand{\heis}[1]{\mathbb H^{#1}}

\newcommand{\smallo}{\operatorname{o}}

\newcommand{\dist}{\operatorname{dist}}

\newcommand{\restrict}[1]{\lfloor_{ #1}}

\newcommand{\spa}{\operatorname{Span}}
\newcommand{\supp}{\operatorname{Spt}}

\theoremstyle{definition}

\newtheorem{definition}{Definition}[section]

\newtheorem{example}[definition]{Example}

\theoremstyle{plain}

\newtheorem{lemma}[definition]{Lemma}
\newtheorem{theorem}[definition]{Theorem}
\newtheorem{corollary}[definition]{Corollary}
\newtheorem{proposition}[definition]{Proposition}

\theoremstyle{remark}

\newtheorem{remark}[definition]{Remark}

\def\Xint#1{\mathchoice
{\XXint\displaystyle\textstyle{#1}}%
{\XXint\textstyle\scriptstyle{#1}}%
{\XXint\scriptstyle\scriptscriptstyle{#1}}%
{\XXint\scriptscriptstyle\scriptscriptstyle{#1}}%
\!\int}
\def\XXint#1#2#3{{\setbox0=\hbox{$#1{#2#3}{\int}$}
\vcenter{\hbox{$#2#3$}}\kern-.5\wd0}}

\def\dashint{\Xint-}

\title[Metric currents and Carnot groups]{Metric currents, differentiable structures, and Carnot groups}
\author{Marshall Williams}

\begin{document}

\thanks{Partially supported under NSF awards \#0602191, \#0353549, and \#0349290.}
\subjclass[2010]{Primary: 30L99; Secondary: 49Q15.}

\address{Department of Mathematics, Statistics, and Computer Science\newline
University of Illinois at Chicago\newline
322 Science and Engineering Offices (M/C 249)\newline
851 S. Morgan Street\newline
Chicago, IL 60607-7045}\email{mcwill@uic.edu}

\begin{abstract}
We examine the theory of metric currents of Ambrosio and Kirchheim in the setting of spaces admitting differentiable structures in the sense of Cheeger and Keith.  We prove that metric forms which vanish in the sense of Cheeger on a set must also vanish when paired with currents concentrated along that set.  From this we deduce a generalization of the chain rule, and show that currents of absolutely continuous mass are given by integration against measurable $k$-vector fields.  We further prove that if the underlying metric space is a Carnot group with its Carnot-Carath\'eodory distance, then every metric current $T$ satisfies $T\restrict{\theta}=0$ and $T\restrict{d\theta}=0$, whenever $\theta \in \diffforms{1}(\BBG)$ annihilates the horizontal bundle of $\BBG$.  Moreover, this condition is necessary and sufficient for a metric current with respect to the Riemannian metric to extend to one with respect to the Carnot-Carath\'eodory metric, provided the current either is locally normal, or has absolutely continuous mass.
\end{abstract}

\maketitle

\input{Introduction} 

\input{Notation}  

\input{Metrickcurrents}  

\input{Strongmeasureddifferentiablestructures}  

\input{Currentsanddifferentiation}  

\input{Metricgroups} 

\input{Carnotgroups} 

\input{Precurrentsincarnotgroups} 

\input{Invariantcurrents}  

\input{Generalcurrentsincarnotgroups} 

\input{Normalcurrentsincarnotgroups}  

\input{Rectifiability}  

\input{References}  








\end{document}

%% file: Introduction.tex
\section{Introduction}

In \cite{AmbrosioKirchheim}, Ambrosio and Kirchheim introduced a definition of currents in metric spaces, extending the theory of normal and integral currents developed by Federer and Fleming \cite{FedererFleming} for Euclidean spaces.  The extension of these classes of currents allows the formulation of variational problems in metric spaces, and the validity of the compactness and closure theorems of \cite{FedererFleming}, proven in the metric setting in \cite{AmbrosioKirchheim}, allows for their solution.  

In this paper we investigate the theory of metric currents in spaces that admit differentiable structures, in the sense of Cheeger \cite{Cheeger} and Keith \cite{Keith}, with a particular emphasis on Carnot Groups equipped with their Carnot-Carath\'eodory metrics.  Most of the results were originally proved in the author's doctoral thesis \cite{Williams}.

\subsection*{Metric currents.}

The classical theory of currents goes back to de Rham \cite{DeRham}.
A current, in the sense of de Rham, is a member of the dual space to the space of smooth differential forms, in analogy with distributions being dual to smooth functions (in fact, distributions are $0$-dimensional currents).  A prototypical example of a $k$-dimensional current in $\BBR^n$ is the map $\omega \mapsto \int_M \omega$, where $M\subseteq \BBR^n$ is an embedded Riemannian submanifold of dimension $k$.  With this example in mind, one defines a boundary operator via Stokes' theorem, in a similar manner to how one differentiates distributions using integration by parts. Likewise, the push-forward of a current along a map is defined through duality by pulling back forms. 

Federer and Fleming studied various classes of currents with finite and locally finite mass \cite{FedererFleming}.  Continuing with the analogy between distributions and currents, one should think of a current of finite mass as being analogous to a measure, and in fact, this can be made precise if one is willing to consider vector valued measures.  The authors of \cite{FedererFleming} introduced the classes of normal currents (currents with finite mass whose boundaries also have finite mass) and integral currents  (normal currents represented by integration along a rectifiable set). 
They then proved a number of compactness and closure theorems, providing new tools for the formulation and solution of area minimization problems in $\BBR^n$, including the well-known Plateau problem. 

Motivated by an idea of De Giorgi \cite{DeGiorgi}, Ambrosio and Kirchheim \cite{AmbrosioKirchheim} extended the Federer-Fleming theory to general metric spaces by replacing the space of smooth forms with a space $\forms{k}(X)$ of Lipschitz $k$-tuples  $(f,g^1,\dotsc,g^k)$, written suggestively as $f\, dg^1\wedge\dotsb\wedge dg^k$.  A metric $k$-current $T\in \currents{k}(X)$ is defined to be a real-valued function on $\forms{k}(X)$ that is linear in each argument, continuous in an appropriate sense, vanishes where it ought to (namely, on forms $f dg^1\wedge\dotsb\wedge dg^k$ such that one of the functions $g^i$ is constant on the support of $f$), and satisfies a finite mass condition.  They demonstrated that most of the results of \cite{FedererFleming} carry over to this more general setting, and that moreover, the classes of classical and metric normal currents are naturally isomorphic in the Euclidean case.  They also proved that rectifiable currents can be classified using the metric and weak* differentiation theorems from the paper \cite{AmbrosioKirchheimRect}, mentioned below. 

Lang \cite{Lang} has introduced a variation of the Ambrosio-Kirchheim theory tailored specifically to locally compact spaces.  In this setting, the finite mass axiom is eliminated.  In spite of this, a number of results from \cite{AmbrosioKirchheim}, including the Leibniz rule and a chain rule, remain true, though the powerful closure and compactness theorems still require assumptions on the masses of currents and their boundaries \cite{Lang}, as is the case for the corresponding results in \cite{AmbrosioKirchheim} and \cite{FedererFleming}.

\subsection*{Differentiable structures}

To formulate the most general of our results below, we will need the notion of a differentiable structure, defined by Keith \cite{Keith}, and motivated by Cheeger's  differentiation theorem.

A (strong measured) differentiable structure on a metric measure space $X$ is a measurable covering of $X$ by coordinate charts $(Y,\pi)$. Here $\pi\colon Y \rightarrow \BBE$ is a Lipschitz map into a Euclidean space $\BBE$.  The defining property of a differentiable structure is the existence, for any Lipschitz function $f$, of a measurable map $y\mapsto \diff{\pi} f_y\in \BBE^*$, satisfying
\begin{equation}
\label{infinitesimalintro}
f(x) = f(y) + \langle \diff{\pi} f_y,\pi(x)-\pi(y)\rangle + \smallo(\dist(x,y))
\end{equation}
at almost every $y\in Y$.  We denote this full measure set of differentiability by $Y_f$.

The differentiation theorems of \cite{Cheeger} and \cite{Keith} state that a nice enough metric measure space $X$ has a countable covering of measurable coordinate patches $X_i$, possibly of different dimensions, on each of which one can differentiate Lipschitz functions.  ``Nice enough'' means, in the case of \cite{Cheeger}, that $X$ has a doubling measure and satisfies a Poincar\'e inequality, as defined in \cite{HeinonenKoskela}.
This differentiation determines a  ``measurable cotangent bundle'' on $X$, which coincides with the usual cotangent bundle if $X$ is a Riemannian Manifold \cite{Cheeger}. 

The theorems of \cite{Cheeger} and \cite{Keith} generalize a number of earlier results.  The classical version of Rademacher's theorem states that a Lipschitz map between Euclidean spaces is differentiable almost everywhere.  Pansu \cite{Pansu} generalized the theorem to maps between Carnot groups, stratified Lie groups equipped with the so-called Carnot-Carath\'eodory metric.  Ambrosio and Kirchheim \cite{AmbrosioKirchheimRect} proved analogs of Rademacher's theorem for maps from Euclidean spaces into general metric spaces, using Banach spaces as an intermediary tool.  Cheeger and Kleiner \cite{CheegerKleiner} have recently  extended the original differentiation theorem from \cite{Cheeger} to Banach space-valued maps.

\subsection*{Metric currents and differentiation.}
We state our main results here, defer some of the more technical definitions to the preliminary material in Sections \ref{preliminaries}, \ref{currentdefinitions}, \ref{diffstructure}, and \ref{carnot}.

Our first results concern the compatibility of the theory of metric currents with the differentiable structures of \cite{Keith}.  The most fundamental of these is a compatibility theorem which states that metric forms that are equivalent in the sense of differentiable structures are also equivalent in the sense of currents, provided the current is concentrated where the forms are defined.

As preliminary notation, let 
\begin{equation*}
\omega=\sum_{s\in S} \beta_s \,dg_s^1 \wedge \dotsb\wedge dg_s^k\in\measextformsc{k}(X)\text{,}
\end{equation*}
where $S$ is finite, and let $Y$ be a coordinate patch.  Here $\measextformsc{k}(X)$ denotes a space of ``measurable metric $k$-forms'', defined precisely in Section \ref{currentdefinitions}, but which should intuitively be thought of as the space of differential forms with compact support and measurable coefficients.

Denote by $Y_\omega\subset Y$ the set of points $y\in Y$ such that all of the functions $g_s^i$, for $i=1,\dotsc,k$, $s\in S$  are differentiable at $y$, and such that
\begin{equation*}
\sum_s \beta(y)\,\diff{\pi} g_{s,y}^1\wedge\dotsb\wedge \diff{\pi} g_{s,y}^k = 0\text{.}
\end{equation*}

We then obtain the following compatibility theorem relating the theories of currents and differentiable structures.
\begin{theorem}
\label{representation}
Let $X = (X,d,\mu)$ be a locally compact metric measure space admitting a strong measured differentiable structure, let $(Y, \pi)$ be a chart, $\pi\colon Y\rightarrow \BBE$,
and let 
\begin{equation*}
\omega\in\measextformsc{k}(X)\text{.}
\end{equation*}

Then for every $T\in \currentsloc{k}(X)$ concentrated on $Y_\omega$, 
\begin{equation}
\label{vanishingcheeger}
T(\omega)=0\text{.}
\end{equation}
\end{theorem}
See Sections \ref{currentdefinitions} and \ref{diffstructure} for more precise definitions.

From Theorem \ref{representation}, we are able to derive a number of results, including the following generalization of the chain rule in \cite{AmbrosioKirchheim} to mappings into arbitrary spaces admitting differentiable structures.

\begin{theorem}
\label{cheegerchainrule}
Let $F\colon Z\rightarrow Y$ be a Lipschitz map, where $(Y,\pi)$ is a coordinate chart.  Let $\beta\in\boundedc(Z)$, and let $(g^1,\dotsc,g^k)\in \liploc(Y)^k$.  Let $Y_{\mathcal G}=\cap_{i=1}^k Y_{g^i}$. Then for any current $T\in \currents{k}(Z)$ such that $F_\#(T\restrict{\beta})$ is concentrated on $Y_{\mathcal G}$,  
\begin{align*}
&T(\beta \,d(g^1\circ F)\wedge\dotsb\wedge d(g^k\circ F))\\
=&  T\left(\sum_{a \in \Lambda_{k,n} }\beta \det\left(\frac{\partial g^i}{\partial \pi^{a_j}}\circ F\right) \,d (\pi^{a_1}\circ F) \wedge \dotsb \wedge d(\pi^{a_k} \circ F)\right)\text{.}
\end{align*}
\end{theorem}
The terminology will be explained precisely in the discussion in Section \ref{representationtheorem} preceding Corollary \ref{repk}.  Loosely speaking, one should think of the functions $\frac{\partial g^i}{\partial \pi^{a_j}}$ as coordinates of the Cheeger differential $\diff{\pi}g^i$, in analogy with partial derivatives in $\BBR^n$.

Theorem \ref{cheegerchainrule} was proved for $Y=\BBR^n$ in \cite[Theorem 3.5 (i)]{AmbrosioKirchheim}, and then generalized \cite[Theorem 2.5]{Lang} to currents that might fail the finite mass axiom.  In both versions, $C^1$-smoothness is assumed for the functions $g^i$, so that $Y_{g^i}=\BBR^n$ by hypothesis.  As a result, there are no hypotheses necessary regarding the measure $||T||$. In contrast, the locally-finite-mass assumption of Theorem \ref{cheegerchainrule} is irremovable; a Lipschitz function need not be differentiable everywhere, and so we must require our current to be concentrated on a set where it is. 

We also prove that certain currents are given by integration against a vector measure. This is already known to be true for all Federer-Fleming currents of locally finite mass \cite[4.1.5]{Federer}, and thus it also holds for all metric currents in $\BBR^n$ as well, via an embedding given in \cite[Theorem 5.5]{Lang}.  In our generality, we are forced to restrict our attention to currents with absolutely continuous mass relative to the underlying measure $\mu$ of the space $X$.

We define a \textit{$k$-precurrent} $T$ to be a functional on the space of metric forms given by integration against a measurable ``$k$-vector field'' $\hat{\lambda}\colon Y \rightarrow \bigwedge^k \BBE$.
\begin{equation*}
T(\beta \,d g^1 \wedge \dotsb \wedge d g^m)  =\int_Y \langle \beta\,\diff{\pi} g^1 \wedge \dotsb \wedge \diff{\pi} g^k,\hat{\lambda}\rangle \, d\mu\text{.}
\end{equation*}
\begin{theorem}
\label{vectormeasure}
Let $X=(X,\dist,\mu)$ be a metric measure space admitting a differentiable structure.  Then every metric $k$-current $T$ in $X$, with $||T||\ll \mu$, is a $k$-precurrent. 
\end{theorem}
The converse of Theorem \ref{vectormeasure} is true in Euclidean space \cite{AmbrosioKirchheim}, but, as we shall see in Theorem \ref{characterization}, precurrents need not be currents in the general case.  

\subsection*{Carnot groups.}

A Carnot group is a simply connected nilpotent Lie group $\BBG$ whose Lie algebra $\mathfrak g$ admits a stratification 
$$\mathfrak g = V_1 \oplus \dotsb \oplus V_n$$
such that $[V_i,V_j] = V_{i+j}$ with the convention that $V_m = 0$ for $m>n$.  The subspace $H=V_1$, together with its left translates, forms what is known as the horizontal bundle.  A theorem of Chow and Rashevsky 
says that any two points in a Carnot group can be joined by a path whose velocity is horizontal at each point.  This leads to a natural definition of a metric on Carnot groups, the so-called Carnot-Carath\'eodory metric $\cdist(p,q)$, given by the shortest horizontal path between $p$ and $q$  (with respect to some invariant Riemannian metric).

The simplest non-Riemannian examples of Carnot groups are the Heisenberg groups $\heis{n}$.  The Lie algebra of $\heis{n}$ is spanned by vector fields $X_1$,\ldots, $X_n$, $Y_1$,\ldots, $Y_n$, and $Z$ satisfying $[X_i,Y_j] = \delta^i_j Z$ and $[X_i,Z]=[Y_j,Z]=0$, and thus admits a stratification 
$$ \spa(X_1,\dotsc,X_n,Y_1,\dotsc,Y_n)\oplus \spa(Z)\text{.}$$
The geometry of $\heis{n}$, equipped with its Carnot-Carath\'eodory metric, is highly non-smooth.  One can show, for example, that its topological dimension is $2n+1$, whereas its Hausdorff dimension is $2n+2$ \cite{Gromov}.  
In spite of this,
Jerison \cite{Jerison} proved that a Carnot group satisfies a Poincar\'e inequality, and so by the result of \cite{Cheeger}, it admits a differentiable structure.  The earlier differentiation theorem of Pansu \cite{Pansu} actually gives an explicit formulation of this structure.  In fact, Cheeger \cite{Cheeger} and Weaver \cite{Weaver} showed that Cheeger's cotangent bundle is given by the dual to the horizontal sub-bundle of the classical tangent bundle.

\subsection*{Currents in Carnot groups.}

The non-commutativity of Carnot groups will prevent certain precurrents from satisfying the continuity axiom for currents.  Our main result characterizes exactly which precurrents are currents in a given Carnot group.

First we must remark that in Carnot groups, there is a natural way to apply metric currents and precurrents to differential forms, as such forms can be written as linear combinations of the form $f\wedge dg^1\wedge \dotsb\wedge dg^k$.  We describe this in more detail in Definitions \ref{smoothformdef} and \ref{smoothrestrictions} (compare also the embeddings in \cite[Theorem 5.5]{Lang} and \cite[Theorem 11.1]{AmbrosioKirchheim}).

Let $\theta\in \diffforms{1}(\BBG)$.  We say that $\theta$ is vertical if $\theta$ annihilates the horizontal bundle of $\BBG$.  For example, the contact form in a Heisenberg group is a vertical form. 
\begin{theorem}
\label{characterization}
Let $\BBG=(\BBG,d_{CC},\mu)$ be a Carnot group, equipped with its Carnot-Carath\-\'eodory metric $d_{CC}$ and Haar measure $\mu$, and let $k\geq 2$.  Then a $k$-precurrent $T$ is a current if and only if 
\begin{equation}
\label{chareq}
T(\theta\wedge\alpha+d\theta\wedge\beta)=0
\end{equation}
whenever $\theta\in\diffforms{1}(\BBG)$ is vertical, $\alpha\in \diffforms{k-1}(\BBG)$, and $\beta\in\diffforms{k-2}(\BBG)$.  Moreover, every current in $T\in\currentsloc{k}(\BBG)$ satisfies equation \eqref{chareq}. 
\end{theorem}

In the case $\BBG=\BBR^n$, the horizontal bundle is the entire space, so that only $0\in\diffforms{1}(\BBR^n)$ is vertical, whereby Theorem \ref{characterization} reduces to a known result \cite[Theorem 3.8]{AmbrosioKirchheim} that every precurrent in $\BBR^n$ is a current.  From the bi-Lipschitz invariance of both the classes of metric currents and precurrents, it follows immediately that the same is true for metric currents in a Carnot group equipped with a \textit{Riemannian} metric (or, for that matter, any rectifiable set equipped with its Hausdorff measure). We therefore have the following corollary to Theorem \ref{characterization}.

\begin{corollary}
\label{absisomorphism}
Let $\BBG_R=(\BBG,\rdist,\mu)$, $\BBG_{CC}=(\BBG,\cdist,\mu)$ be the Carnot group $\BBG$ equipped, respectively, with Riemannian and Carnot-Carath\'eodory distances, and Haar measure $\mu$.  Then the pushforward map $I_{\#}\colon \currentsloc{k}(\BBG_{CC})\rightarrow\currentsloc{k}(\BBG_R)$ induced by the identity map $I\colon \BBG_{CC}\rightarrow\BBG_R$ restricts to an isomorphism
\begin{equation*}
I_\#|_{\currentslocabs{k}(\BBG_{CC})}\colon \currentslocabs{k}(\BBG_{CC})\rightarrow \currentslocabsvanishing{k}(\BBG_R)\text{,}
\end{equation*}
where
\begin{equation*}
\currentslocabs{k}(\BBG_{CC})=\{T\in\currentsloc{k}(\BBG_{CC}): ||T||\ll \mu\}
\end{equation*}
and
\begin{equation*}
\currentslocabsvanishing{k}(\BBG_R)=\{T\in\currentsloc{k}(\BBG_R): ||T||\ll \mu\text{ and $T$ satisfies \eqref{chareq}.}\}
\end{equation*}
\end{corollary}

We do not know if the condition of absolute continuity of the mass can be removed entirely from Corollary \ref{absisomorphism}.  However, we are able to replace it with normality.  A metric current $T\in\currentsloc{k}(X)$ is \textit{locally normal} (written $T\in\normalloc{k}(X)$) if its boundary has locally finite mass, i.e., $\partial T\in\currentsloc{k}(X)$.  Let
\begin{equation*}
\normallocvanishing{k}(\BBG_R)=\{T\in\normalloc{k}(\BBG_R): \text{$T$ satisfies \eqref{chareq}.}\}
\end{equation*}
\begin{theorem}
\label{normalisomorphism}
The pushforward map $I_{\#}\colon \currentsloc{k}(\BBG_{CC})\rightarrow\currentsloc{k}(\BBG_R)$ 
restricts to an isomorphism $I_\#|_{\normalloc{k}(\BBG_{CC})}\colon \normalloc{k}(\BBG_{CC})\rightarrow \normallocvanishing{k}(\BBG_R)$.
\end{theorem}

\subsection*{Applications to Heisenberg Groups.}
Theorem \ref{characterization} and Corollary \ref{absisomorphism} have a number of implications when applied to Heisenberg groups. 
As observed by Rumin \cite[Section 2]{Rumin}, the space $\diffforms{k}(\heis{n})$ of smooth $k$-forms, for $k\geq n$, consists entirely of forms $\alpha\wedge\theta +\beta\wedge d\theta$, where $\theta$ is the contact form.  From this, Theorem \ref{characterization}, and a density result for smooth forms (Corollary \ref{smootheval}), we obtain a bound for the dimension of a nonzero metric current in $\heis{n}$.
\begin{corollary}
\label{nohncurrents}
Let $k>n$.  Then $\currentsloc{k}((\heis{n},\cdist))=0$.
\end{corollary}
Corollary \ref{nohncurrents} generalizes a result of Ambrosio and Kirchheim \cite[Theorem 7.2]{AmbrosioKirchheimRect} that says that nonzero $k$-\textit{rectifiable} currents do not exist in $\heis{n}$ for $k>n$. We discuss the issue of rectifiability in Section \ref{rectifiability}, where we explore the relationship between our results and those of \cite{AmbrosioKirchheimRect}, as well as a more general result due to Magnani \cite{Magnani}. 

Franchi, Serapioni, and Serra Cassano \cite{FSSC} have recently  developed an extension of the Federer-Fleming theory to Heisenberg groups.  For $k\leq n$, equation \eqref{chareq} is a defining property of their ``Heisenberg currents''.   This suggests that metric currents, rectifiable or not, might best be thought of as fundamentally low-dimensional objects.

\subsection*{Organization of the paper.}
Section \ref{preliminaries} establishes basic notation.  In Section \ref{currentdefinitions} we review the basic facts from the theory of currents as developed in \cite{AmbrosioKirchheim} and \cite{Lang}, introducing some variations in the notation necessary for our purposes.  In Section \ref{diffstructure} we recall the notion of a differentiable structure as introduced by \cite{Keith}, motivated by the differentiation theorems of \cite{Cheeger}.  The definitions and notation in these sections make precise the statements in Theorems \ref{representation}, \ref{cheegerchainrule}, and \ref{vectormeasure}, which we  proceed to prove in Section \ref{representationtheorem}.  Section \ref{metricgroups} briefly discusses the behavior of metric currents in a group setting,  establishing a density result for currents of absolutely continuous mass (Proposition \ref{approximationlemma}).  For the remainder of the article, we restrict our attention to Carnot groups.  After reviewing the important facts from the theory in Section \ref{carnot}, we perform in Section \ref{carnotprecurrents} an analysis of precurrents in such groups.  In Section \ref{invariantcurrents}, we prove Proposition \ref{currentsk}, a special case of Theorem \ref{characterization}.  The general result is then proved in Section \ref{generalcarnotcurrents}.  We prove Theorem \ref{normalisomorphism} in Section \ref{normalsurjection}, and conclude with a discussion relating our results to previously known rectifiability theorems in Section \ref{rectifiability}.



\subsection*{Acknowledgments}

I wish to thank my advisor, Mario Bonk, as well as Pekka Pankka and Stefan Wenger, for reviewing several drafts of this work, and providing numerous comments and suggestions.  I am also grateful for additional comments from Luigi Ambrosio, Urs Lang, and Valentino Magnani.  Most especially, I am deeply indebted to my late advisor, Juha Heinonen, who introduced me to this field, offering much insight, patient guidance, and encouragement.

%% file: Notation.tex
\section{Notation}
\label{preliminaries}
Throughout this paper, $X = (X,\dist)$ will denote a separable, locally compact metric space.
We will frequently make use of the notation $|x_1-x_2|=\dist(x_1,x_2)$ when $x_1, x_2\in X$, and the metric is unambiguous. The closed ball of radius $r$ centered at a point $x_0\in X$ is denoted by $B_r(x)$

A Euclidean space is a finite dimensional vector space whose metric is given by an inner product.  Typically, we will use the notation $\BBE$ to refer to a Euclidean space of unspecified dimension, as well as $\mathfrak e$ to denote the Lie algebra associated to the Lie group $(\BBE,+)$. 

The term ``function'' will always denote a real valued map, and we will denote the support of a function $f$ by $\supp(f)$.  This is defined to be the smallest closed set outside of which $f$ vanishes. 

The Lipschitz constant of a map $F\colon X\rightarrow Y$ is  denoted by $L(F)$
We write $\lipc(X)$, $\liploc(X)$, and $\lip_1(X)$ to denote, respectively, the spaces of Lipschitz functions with compact support, locally Lipschitz functions, and functions with Lipschitz constant at most $1$.  We equip these spaces  with notions of convergence, though we do not define topologies on them, since we are interested only in convergence of sequences, and knowledge of such convergence is not generally sufficient to describe a vector space topology.
Instead, we simply say that a sequence of functions $f_i\in\lipc(X)$ \textit{converges} to $f\in\lipc(X)$ if the sequence converges to $f$ pointwise, and all of the functions $f_i$ and $f$ have uniformly bounded Lipschitz constant, as well as uniformly compact support.  Similarly, a sequence $f_i \in \liploc(X)$ converges to $f\in \liploc(X)$ if it converges to $f$ pointwise, and for any compact subset $K\subset X$, all of the functions $f_i$ and $f$ have uniformly bounded Lipschitz constants when restricted to $K$.  
Though we do not discuss them here, \cite{Lang} describes locally convex vector space topologies which yield the same notion of convergent sequences.  

We say that a subset $S\subset \lipc(X)$ is \textit{dense} if every function in $\lip_c(X)$ is a limit of a sequence of functions in $S$.  Note that since $X$ is locally compact and separable, each subset $S_K^M$, where
\begin{equation*}
S_K^M = \{f\in\lipc(X): L(f)<M\text{, }|f(x)|<M\text{ for all }x\in X\text{, and }\supp(f)\subseteq K \text{.}\}\text{,}
\end{equation*}
is compact (and hence separable) in the topology of uniform convergence, by the Arzela-Ascoli Theorem.  Here $M>0$ and $K\subset X$ is compact.  It follows that $\lipc(X)$ is \textit{separable} in the sense that it has a countable dense subset. 

We denote by $\bounded(X)$, $\boundedc(X)$ and $\boundedloc(X)$ the spaces of Borel functions that are, respectively, bounded, bounded with compact support, and locally bounded. 

If $\mu$ is a Borel measure on a space $X$, and $F\colon X\rightarrow Y$ is Borel measurable, then $F_\#\mu$ denotes the pushforward of $\mu$ by $F$; that is, $F_\#\mu$ is the Borel measure on $Y$ given by $F_\#\mu(A) = \mu(F^{-1}(A))$ for every Borel set $A \subseteq Y$.

By $\mathcal H^k(A)$ we denote to the $k$-dimensional Hausdorff measure of a subset $A\subseteq X$.  If $V$ is a vector space, $\bigwedge^k V$ is the $k^{\text{th}}$ exterior power of $V$.  Finally, we denote by $\Lambda_{k,n}$ the set of $k$-indices of the form $(i_1,\dotsc,i_k)$ satisfying $1 \leq i_i <\dotsb<i_k \leq n$. 

%% file: Metrickcurrents.tex
\section{Metric $k$-currents}

\label{currentdefinitions}

Let $X$ be a locally compact metric space.  We recall a few definitions from \cite{Lang}.  We will follow \cite{Lang} throughout this section, except as noted otherwise.  One small addition will be our linearization of the spaces of ``forms'' via tensor products and exterior powers, as described below.

First we define the space $\formsc{k}(X)$ of compactly supported simple metric $k$-forms by
\begin{equation*}
\formsc{k}(X) = \lipc(X) \times \liploc(X)^k\text{.}
\end{equation*}
The motivation for calling elements of this space ``simple forms'' will be explained below.
We say that a sequence of $k$-forms $\omega_i = (f_i,g_i^1,\dotsc,g_i^k)$ converges to $\omega = (f,g^1,\dotsc,g^k)$ if $f_i$ converges to $f$ and $g_i^j$ converges to $g^j$ for $j=1,\dotsc,k$.  Here and throughout, unless otherwise stated, the convergence of a sequence of Lipschitz functions is defined as in Section \ref{preliminaries}.

We also define a number of other spaces, in which we will not concern ourselves with notions of convergence: 
\begin{equation*}
\extformsc{k}(X)= \lipc(X) \otimes \bigwedge\nolimits^k\liploc(X)\text{.}
\end{equation*}
\begin{equation*}
\forms{k}(X)  = \liploc(X)^{k+1} \text{.}
\end{equation*}
\begin{equation*}
\extforms{k}(X)  = \liploc(X)\otimes\bigwedge\nolimits^k\liploc(X)\text{.}
\end{equation*}
\begin{equation*}
\measforms{k}(X) = \boundedloc(X) \times \liploc(X)^{k}\text{.}
\end{equation*}
\begin{equation*}
\measextforms{k}(X) = \boundedloc(X) \otimes \bigwedge\nolimits^k\liploc(X)\text{.}
\end{equation*}
\begin{equation*}
\measformsc{k}(X) = \boundedc(X) \times \liploc(X)^{k}\text{.}
\end{equation*}
\begin{equation*}
\measextformsc{k}(X) = \boundedc(X) \otimes \bigwedge\nolimits^k\liploc(X)\text{.}
\end{equation*}

\begin{remark}
The number of different spaces of ``forms'' may at first appear daunting, but we do not require deep study for most of them.  As stated before, we do not topologize any of these additional spaces - any time we speak of convergence of a sequence of forms, we \textit{always} refer to a sequence of simple forms in $\formsc{k}(X)$.

Our use of tensor and exterior products here is a deviation from both \cite{AmbrosioKirchheim} and \cite{Lang}.  The motivation for this is two-fold.  Philosophically, in order to complete the analogy between ``metric forms'' and classical differential forms, we would like for metric forms to constitute a linear space.  More practically, in our formulation and proof of Theorem \ref{characterization}, we need to deal with metric forms that are not simple.  However, it should be noted that this deviation from the theory is entirely cosmetic, due to our lack of topological considerations on any of the additional spaces.  We use them only to more naturally phrase statements that would otherwise require repeated discussion of linear combinations of forms.
\end{remark}

\begin{definition}
\label{currentdef}
A \textbf{metric $k$-current} on $X$ is a map $T\colon\formsc{k}(X) \rightarrow \BBR$ satisfying the following axioms:
\begin{enumerate}
\item Linearity:  $T$ is linear in each argument.
\item Continuity:  $T(\omega_i)$ converges to $T(\omega)$ whenever $\omega_i$ converges to $\omega$.
\item Locality:  $T((f,g^1,\dotsc,g^k))=0$ provided that for some $i$, $g^i$ is constant on $\supp(f)$.
\end{enumerate}
The space of metric $k$-currents on $X$ is denoted $\lcurrents{k}(X)$.
\end{definition}
We will frequently drop the adjective ``metric'' in the future. 
\begin{remark}
\label{localityremark}
We should point out that a priori the locality axiom as defined in \cite[Definition 2.1]{Lang} is only required to hold when $g^i$ is constant on a \textit{neighborhood} of $\supp(f)$, but it is later proven there that this is equivalent to the above definition.  Also, as a consequence of the locality axiom, we may modify any of the functions $g^i$ away from $\supp(f)$, or in turn modify $f$ away from $\supp (g^i)$, without changing the value of $T((f,g^1,\dotsc,g^k))$  (to see that the second statement is true, note that $f$ vanishes on a neighborhood of $\supp(g^i)$ if and only if $g^i$ vanishes on a neighborhood of $\supp(f)$).  In particular, if $(f,g^1,\dotsc,g^k)\in\forms{k}(X)$, and one of the functions $g^i$ is $\lipc(X)$, we may unambiguously define
\begin{equation*}
T((f, g^1,\dotsc, g^k))= T((\sigma f,g^1,\dotsc,g^k))\text{,}
\end{equation*}
where $\sigma\in\lipc(X)$ is any function satisfying $\sigma\equiv 1$ on some neighborhood of $\supp(g^i)$. 
\end{remark}

The following theorem provides some intuition for the use of the term ``form'' above.
\begin{theorem}
{\cite[Proposition 2.4]{Lang}}       
\label{alternateleibniz}
If $T\colon \formsc{k}(X)\rightarrow \BBR$ is a $k$-current, then $T$ satisfies the alternating property and the Leibniz rule:
\begin{equation}
\label{alternatingeq1}
T((f, g^1,\dotsc,g^i,\dotsc, g^j,\dotsc,g^k)) = - T((f, g^1,\dotsc,g^j,\dotsc, g^i,\dotsc,g^k))\text{.}
\end{equation}
\begin{equation}
\label{leibnizeq1}
T((f, g^1,\dotsc,g^k)) + T((g^1,f,\dotsc,g^k)) = T((1,f g^1,g^2,\dotsc,g^k))\text{.}
\end{equation}
\end{theorem}
Notice that the right hand side of equation \eqref{leibnizeq1} is well-defined by Remark \ref{localityremark}.

Although we are using the definition of currents from \cite{Lang}, in light of Theorem \ref{alternateleibniz} we will borrow the suggestive notation
\begin{equation}
\label{convention1}
f\,dg^1\wedge \dotsb \wedge dg^k = (f,g^1,\dotsc,g^k)
\end{equation}
from \cite{AmbrosioKirchheim}. 

Moreover, if one of the functions $g^i$ is compactly supported, we define
\begin{equation}
dg^1\wedge \dotsb \wedge dg^k = (1,g^1,\dotsc,g^k)\text{.}
\end{equation}
This latter notation is justified by Remark \ref{localityremark}, and the locality property.

With this new notation, equations \eqref{alternatingeq1} and \eqref{leibnizeq1} can be rewritten:
\begin{equation}
\label{alternatingeq}
T(f\, dg^1\wedge \dotsb \wedge dg^i\wedge\dotsb\wedge dg^j\wedge\dotsb \wedge dg^k) = - T(f\, dg^1\wedge \dotsb \wedge dg^j\wedge\dotsb\wedge dg^i\wedge\dotsb \wedge dg^k)
\end{equation}
\begin{equation}
\label{leibnizeq}
T(f\,dg^1\wedge \dotsb \wedge dg^k) + T(g^1\,df\wedge dg^2\wedge\dotsb\wedge dg^k) = T(d(f g^1)\wedge dg^2\wedge\dotsb\wedge dg^k)\text{.}
\end{equation}

 
Since $T$ is linear in each variable, and satisfies the alternating property \eqref{alternatingeq}, there is a unique linear map, which we also denote by $T\colon\extformsc{k}(X)\rightarrow \BBR$, satisfying $T(f\otimes g^1\wedge\dotsb\wedge g^k)=T(f\,dg^1\wedge\dotsb\wedge dg^k)$.  We will therefore use the notation
\begin{equation}
\label{convention2}
f\,dg^1\wedge \dotsb \wedge dg^k = f\otimes g^1\wedge \dotsb\wedge g^k\text{.}
\end{equation}
Since $T(f\otimes g^1\wedge\dotsb\wedge g^k)=T((f,g^1,\dotsc,g^k))$, there is no potential for ambiguity between the notations introduced with equations \eqref{convention1} and \eqref{convention2};  the only situations in which we consider metric forms involve pairing the forms with currents, with the one exception being that we at times discuss convergence of forms in their own right.  In this latter context, we only deal with convergence of simple forms in $\formsc{k}(X)$, and so in such a setting, we assume the forms are in that space, rather than $\extformsc{k}(X)$.

With our introduction of the space $\extformsc{k}(X)$, we are able to rephrase the definition of mass from \cite{Lang} (\cite[Definition 4.1]{Lang}, but see also \cite[equation (3.7)]{AmbrosioKirchheim}).  We first make a definition that is somewhat reminiscent of the usual notion of comass for differential forms.
\begin{definition}
\label{comass}
Let $\omega \in \extformsc{k}(X)$.  The \textbf{comass} of $\omega$, written $||\omega||$, is the number
\begin{equation*}
||\omega||=\inf_{\text{$S$ finite}} \sum_{s\in S} |f_s|
\end{equation*}
where the functions $f_s$ satisfy
\begin{equation*}
\omega = \sum_{s\in S} f_s\, dg^1_s \wedge \dotsb \wedge dg^k_s
\end{equation*}
for some functions $g^i_s\in\liploc(X)$ such that $L(g^i_s|_{\supp(f)})\leq 1$.
\end{definition}

We now give our reformulation of \cite[Definition 4.1]{Lang}.
\begin{definition}
\label{massdef}
Let $T\colon \formsc{k}(X)\rightarrow \BBR$ be any function that is linear in each argument.  The \textbf{mass} of $T$ is the Borel regular outer measure $||T||$ on $X$ given on open sets $U$ by   
\begin{equation*}
||T||(U) = \sup_{\omega\in\extformsc{k}(U), ||\omega||\leq 1} T(\omega)\text{,}
\end{equation*}
and on arbitrary sets $A$ by
\begin{equation*}
||T||(A) = \inf_{U \supset A\text{, $U$ open}} ||T||(U)\text{.}
\end{equation*}
\end{definition}
It follows from \cite[Theorem 4.3]{Lang} (and the succeeding remarks) that $||T||$ is indeed a Borel regular outer measure.  Notice that we do not require any continuity restrictions on $T$.

We denote by $\currents{k}(X)$ (resp.\ $\currentsloc{k}(X)$) the space of metric $k$-currents of (resp.\ locally) finite mass, that is, 
\begin{equation*}
\currents{k}(X)=\{T\in\lcurrents{k}(X):||T||(X)<\infty\}\text{,}
\end{equation*}
and
\begin{equation*}
\currentsloc{k}(X)=\{T\in\lcurrents{k}(X):\text{$||T||(A)<\infty$ whenever $\overline{A}\subset X$ is compact.}\}\text{.}
\end{equation*}
It can be shown \cite[Proposition 4.2]{Lang} that $\currents{k}(X)$ is a Banach space under the mass norm $||T||(X)$. 


We recall \cite[Theorem 4.4]{Lang} that for every $k$-current $T$ of locally finite mass, there is a canonical extension of $T$ to $\measformsc{k}(X)$, and hence to $\measextformsc{k}(X)$, such that if $f_i\in\lipc(X)$, $\beta\in \boundedc(X)$, and $\{f_i\}$ converges to $\beta$ in $L^1(X, ||T||)$, then for every ordered $k$-tuple $(g^1,\dotsc,g^k)\in \liploc(X)^k$,
\begin{equation}
\label{finitemassextension}
T(\beta\,dg^1\wedge \dotsb\wedge dg^k)= \lim_{i\rightarrow\infty} T(f_i\,dg^1\wedge \dotsb\wedge dg^k)\text{.}
\end{equation}


The mass measure $||T||$ can be characterized alternatively \cite[Theorem 4.3]{Lang} as the minimal Borel regular measure satisfying
\begin{equation}
\label{masscriterion1}
T(f\,dg^1\wedge\dotsb\wedge dg^k)\leq\prod_{i=1}^k L(g^i|_{\supp(f)}) \int_X |f|\,d||T||
\end{equation}
for every $f\,dg^1\wedge\dotsb\wedge dg^k\in \formsc{k}(X)$.

From the definition, the extension of $T$ to $\measformsc{k}(X)$ also satisfies equation \eqref{masscriterion1}.
\begin{equation}
\label{masscriterion1meas}
T(\beta\,dg^1\wedge\dotsb\wedge dg^k)\leq\prod_{i=1}^k L(g^i) \int_X |\beta|\,d||T||\text{.}
\end{equation}

In the case $k=0$, currents of locally finite mass act on functions by integration against a signed Radon measure, absolutely continuous with respect to $||T||$.  The following lemma, and proof, were communicated to the author by Urs Lang.
\begin{lemma}
\label{0currents}
Let $T\in\currentsloc{0}(X)$.  Then there is a function $\lambda\in L^\infty(X,||T||)$ such that for every $\beta\in\measformsc{0}(X)=\boundedc(X)$,
\begin{equation*}
T(\beta)=\int_X \beta \lambda\,d||T||\text{.}
\end{equation*}
\end{lemma}

\begin{proof}
The mapping $T$ is continuous in the norm of $L^1(X,||T||)$, by inequality \eqref{masscriterion1}.  Moreover, the compactly supported Lipschitz functions are dense in this norm, and so $T$ extends to a map $\hat{T}\in L^1(X,||T||)^{*}$  (This is, in fact, precisely the argument used in \cite{Lang} to define the extension \eqref{finitemassextension} above).  Thus the existence and uniqueness of $\lambda$ follows from the Riesz representation theorem.
\end{proof}

As with the classical definition, the boundary of a current is defined through duality:
\begin{definition}
\label{boundarydef}
Let $T\colon\formsc{k}(X)\rightarrow \BBR$ be a $k$-current.  The \textbf{boundary} of $T$ is the map $\partial T\colon\formsc{k-1}(X)\rightarrow \BBR$ given by
\begin{equation*}
\partial T(f\, dg^1 \wedge \dotsb \wedge dg^{k-1}) = T(df \wedge dg^1 \wedge \dotsb \wedge dg^{k-1})\text{.}
\end{equation*}
\end{definition}
As noted after \cite[Definition 3.4]{Lang}, the map $T\mapsto \partial T$ is well-defined, linear in each argument, and satisfies $\partial(\partial T) = 0$. 



Typically we do not expect the boundary of a current to have finite mass.  If a current and its boundary \textit{do} each have finite (resp. locally finite) mass, the current is said to be a normal (resp. locally normal) current.  The space of such currents will be denoted $\normal{k}(X)$ (resp. $\normalloc{k}(X)$). 

Though one of the highlights of \cite{Lang} is the elimination of the assumption of finite mass, or even locally finite mass, as a necessary axiom for the theory of currents, all of the currents we consider from now on will have locally finite mass.  Indeed, our motivation for following \cite{Lang} rather than \cite{AmbrosioKirchheim} is primarily that the former allows for locally finite mass, rather than just finite mass.  For this reason, we introduce the following convention:

\emph{Throughout the rest of this paper, except where otherwise noted, the word ``current'' will denote a metric current of locally finite mass.}

Given a $k$-current and a $j$-form, with $0\leq j\leq k$, there is a natural way to produce a $(k-j)$-current.
\begin{definition}
\label{restriction}
Let $T \in \currentsloc{k}(X)$, and $\omega = \beta \, dh^1 \wedge \dotsb \wedge dh^j \in \measforms{j}(X)$. The restriction of the current $T$ by the form $\omega$ is the $(k-j)$-current $T\restrict{\omega} \in \currentsloc{k-j}(X)$, given by
\begin{equation}
\label{restrictioneq}
T\restrict{\omega}(f\,dg^1 \wedge \dotsb \wedge dg^{k-j})
= T(\beta f\,dh^1 \wedge \dotsb \wedge dh^j \wedge dg^1 \wedge \dotsb \wedge dg^{k-j})\text{.}
\end{equation}
If $A \subseteq X$ is a Borel set, we define
\begin{equation*}
T\restrict{A}=T\restrict{\charfcn_A}\text{.}
\end{equation*}
\end{definition}

Note that for a fixed current $T\in\currents{k}$, the restriction map $T\restrict{}\colon \measforms{j}(X) \rightarrow \currentsloc{k-j}(X)$ is linear in each argument, and thus induces a linear map $T\restrict{}\colon \measextforms{j}(X) \rightarrow \currentsloc{k-j}(X)$.

It can be shown \cite[Lemma 4.7]{Lang} that $||T\restrict{A}||=||T||\restrict{A}$.  Using restrictions, we also have notions of concentration and support for currents.
\begin{definition}
\label{support}
We say that $T$ is \textbf{concentrated} on a Borel set $A\subseteq X$ if $T\restrict{A}=T$, or equivalently, if $||T||$ is concentrated on $A$. The \textbf{support} of a current $T$, denoted $\supp(T)$, is the smallest closed set on which $T$ is concentrated.
\end{definition}

Definition \ref{support} lets us update equation \eqref{masscriterion1meas}:
\begin{equation}
\label{masscriterion}
T(\beta\,dg^1\wedge\dotsb\wedge dg^k)\leq\prod_{i=1}^k L(g^i|_{\supp(T)}) \int_X |\beta|\,d||T||\text{.}
\end{equation}



We recall the notion of the push-forward of a current.

\begin{definition}
\label{pushforwarddef}
Let $F\colon X\rightarrow Y$ be a proper Lipschitz map between metric spaces $X$ and $Y$.  The \textbf{push-forward} of a current $T \in \currentsloc{k}(X)$ along $F$ is the current $F_\#T\in\currentsloc{k}(Y)$ given by
\begin{equation*}
F_\#T(f \,dg^1\wedge \dotsb \wedge dg^k)= T((f\circ F)\,d(g^1\circ F)\wedge \dotsb \wedge d(g^k\circ F))\text{.}
\end{equation*}
\end{definition}

\begin{remark}
\label{pushforwardcompact}
If $T$ is compactly supported, we may drop the assumption that $F$ is proper.  Indeed, in this case we define
\begin{equation*}
F_\#T(f \,dg^1\wedge \dotsb \wedge dg^k)= T(\sigma\cdot(f\circ F)\,d(g^1\circ F)\wedge \dotsb \wedge d(g^k\circ F))\text{,}
\end{equation*}
where $\sigma$ is any compactly supported Lipschitz function such that $\sigma|_{\supp(T)}\equiv 1$.  It follows immediately from the definition of $\supp(T)$ that this is well-defined, and coincides, in the case of a proper map, with the Definition \ref{pushforwarddef}.
\end{remark}


%% file: Strongmeasureddifferentiablestructures.tex
\section{Strong measured differentiable structures}
\label{diffstructure}

We recall the notion of a differentiable structure from  \cite{Keith}, inspired by \cite{Cheeger}.
\begin{definition}
\label{structure}
Let $X = (X,d,\mu)$ be a metric measure space.
Let $Y \subseteq X$, and let $\pi = (\pi^1,\dotsc,\pi^n)\colon Y \rightarrow \BBE$ be a Lipschitz map, where $\BBE\cong \BBR^n$ is a Euclidean space, and $\mathfrak{e}$ it's tangent space.  We call $(Y,\pi)$ a \textbf{coordinate patch} if the following holds:
For any $f \in \lip(X)$ there is a set $Y_f\subset Y$, with $\mu(Y\backslash Y_f)=0$, for which there is a unique measurable function $\diff{\pi} f \colon Y_f \rightarrow \mathfrak{e}^*$ (written $y\mapsto \diff{\pi} f_y$), such that for every $y \in Y_f$, 
\begin{equation}
\label{infinitesimal}
f(x) = f(y) + \langle \diff{\pi} f_y,\pi(x)-\pi(y)\rangle + E^f_y(x)\text{,}
\end{equation}
where 
\begin{equation}
\label{smalloerror}
\lim_{x\rightarrow y}\frac{E^f_y(x)}{\dist(x,y)} = 0\text{.}
\end{equation} 
$X$ admits a \textbf{strong measured differentiable structure} if $X$ is a countable union of coordinate patches. 
\end{definition}
\begin{remark}
\label{liealgebranote}
It is important to note that the Euclidean inner product $\langle,\rangle$ on $\BBE$, induces, via the exponential map, a natural pairing $\langle,\rangle$ between $\mathfrak{e}^*$ and $\BBE$, through which equation \eqref{infinitesimal} should be interpreted.  The reason for such care in distinguishing between Lie groups and Lie algebras will be more apparent when we discuss Carnot groups.
\end{remark}

I am greatly indebted to Stefan Wenger for suggesting the following useful fact, which has strengthened the result of Theorem \ref{representation} while at the same time simplifying its proof (also, compare \cite[Section 3]{AmbrosioKirchheimRect}).
\begin{lemma}
\label{uniformcontrol}
Let $X$ be locally compact and separable, $Y\subset X$, and let $\pi\colon Y\rightarrow \BBE$ be a coordinate patch as in Definition \ref{structure}.  Let $f \in \lipc(X)$, $\epsilon>0$, and let $\nu$ be a Radon measure concentrated on  $Y_f$.  Then there is a compact set $Z=Z(f,\nu,\epsilon)\subseteq Y_f$, with $\nu(Y\backslash Z)<\epsilon$, such that as $r$ approaches $0$, the Lipschitz constant $L(E^f_{z}|_{Z\cap B_r(x)})$ of the restricted error function $E^f_{z}|_{Z\cap B_r(x)}$ converges to $0$ uniformly in $z$, for $z\in Z$.  That is, there is a continuous function $\eta\colon [0,\infty)\rightarrow [0,\infty)$, with $\eta(0)=0$, such that for all $z\in Z$ and $r\in\BBR$,
\begin{equation}
\label{unifcontroleq}
L(E^f_z|_{Z\cap B_r(x)}) \leq \eta(r)\text{.}
\end{equation}
\end{lemma}

\begin{proof}
Consider the functions $E_r\colon Y_f\rightarrow \BBR$ given by
\begin{equation*}
E_r(y) = \sup_{x\in B_r(y), x\neq y} \frac{E^f_y(x)}{\dist(x,y)}\text{.}
\end{equation*} 
Let $S$ be a countable dense subset of $X$ (which exists by the separability of $X$), and observe that for any point $y_0\in Y_f$, the function $E^f_{y_0}$ is continuous (and therefore measurable), by the continuity of the remaining terms in equation \eqref{infinitesimal}.   We therefore have 
\begin{equation*}
E_r(y) = \sup_{x\in S, x\neq y}\{\frac{E^f_y(x)}{\dist(x,y)}\charfcn_{B_r(y)}\}\text{.} 
\end{equation*}
Thus $E_r$ is the supremum of a countable family of measurable functions, and is therefore measurable.

By equation \eqref{smalloerror}, the functions $E_r$ converge to $0$ pointwise on $Y_f$.  Thus by Egorov's Theorem, there is a subset $Z_1\subset Y_f$, with $\nu(Y_f\backslash Z_1)\leq \epsilon/3$, on which the functions $E_r$ converge uniformly.  That is, there is a continuous function $\eta_1\colon [0,\infty)\rightarrow [0,\infty)$, with $\eta_1(0)=0$, such that $E_r(z)\leq \eta_1(r)$ for all $z\in Z_1$.
On the other hand, by Lusin's theorem, the measurability of the function $\diff{\pi} f$ guarantees the existence of a subset $Z_2\subset Y_f$, with $\nu(Y_f\backslash Z_2)\leq\epsilon$, on which $\diff{\pi} f$ is uniformly continuous, i.e., there is a continuous function $\eta_2\colon [0,\infty)\rightarrow [0,\infty)$, with $\eta(0)=0$, such that $||\diff{\pi} f_x - \diff{\pi} f_y||\leq \eta_2(\dist(x,y))$.

Let $Z=Z_1\cap Z_2$.  Then $\mu(Y_f\backslash Z)<\epsilon$, and for every $z\in Z$ and every $x,y\in B_r(z)\cap Z$, $x\neq y$, we have
\begin{align*}
&\left|\frac{E^f_z(x)-E^f_z(y)}{\dist(x,y)}\right|= \left|\frac{f(x) - f(y) - \langle \diff{\pi} f_z,\pi(x)-\pi(y) \rangle}{\dist(x,y)}\right|\\
\leq& \frac{\left|f(x) - f(y) - \langle \diff{\pi} f_y, \pi(x)-\pi(y) \rangle\right| +\left|\langle \diff{\pi} f_z - \diff{\pi} f_y, \pi(x)-\pi(y)\rangle\right|}{\dist(x,y)}\\
\leq& E_r(y) + \eta_2(r)\frac{||\pi(x)-\pi(y)||}{\dist(x,y)} \leq \eta_1(r) + \eta_2(r)L(\pi)\text{.}
\end{align*}
Letting $\eta(r)=\eta_1(r) + \eta_2(r)L(\pi)$ completes the proof.
\end{proof}

%% file: Currentsanddifferentiation.tex
\section{Currents and differentiation}
\label{representationtheorem}
In this section, we prove Theorems \ref{representation} and  \ref{vectormeasure}, as well as some other useful results for relating metric currents and differentiable structures.
All of our other results in this section stem from Theorem \ref{representation}, which we now prove.

\begin{proof}[Proof of Theorem \ref{representation}]
Fix $T\in \currentsloc{k}(X)$ with $||T||$ concentrated on $Y_\omega$.  
We assume with no loss of generality that $L(g_s^i)\leq 1$ for $i=1,\dotsc,k$ and all
$s\in S$, that $|\beta_s(x)|\leq 1$ for every $x\in X$ and $s\in S$, and that $L(\pi)=1$.
It is enough to show that for every $\epsilon>0$, equation \eqref{vanishingcheeger} holds when $T$ is replaced with $T\restrict{Z}$, where  
\begin{equation*}
Z = \left(\bigcup_s \supp(\beta_s)\right) \cap \bigcap_{s,i} Z(g_s^i,||T||,\epsilon)\text{,}
\end{equation*}
and where each set $Z(g_s^i,||T||,\epsilon)$  is chosen as in Lemma \ref{uniformcontrol}, so that for every $z\in Z$, each restricted error function $E^{g^i_s}_z|_{B_r(z)}$ has Lipschitz constant $L(E^{g^i_s}_z|_{B_r(z)})<\eta(r)$.
Indeed, if this is the case, then by the mass criterion \eqref{masscriterion}, we have
\begin{equation*}
|T(\omega)|=|T\restrict{Y_f\backslash Z}(\omega)| \leq k\#S\cdot||T||(Y_f\backslash Z)\leq k\#S \cdot\epsilon\text{,}
\end{equation*}
from which the result follows upon passing to the limit as $\epsilon$ approaches $0$.

By the remarks in the previous paragraph, we may assume without loss of generality that $T=T\restrict{Z}$. Further, we will assume that for each $i$ and $s$, $||\diff{\pi}g^i_s||\leq 1$ on $Z$, where $||\cdot||$ is the dual norm to the Euclidean norm on $\BBE$.  This is a harmless assumption, as the differentials are measurable, and hence bounded by some number $M$ away from a set of arbitrarily small $||T||$-measure on $Z$. Rescaling the functions allows us to assume $M=1$.  Notice that under this assumption, for all $z\in Z$, the Lipschitz constants of the functions $y\mapsto \langle \diff{\pi}g^i_{s,z},\pi(y)\rangle$ are at most $1$, that is,
\begin{equation}
\label{bracketbound}
L(\langle \diff{\pi}g^i_{s,z},\pi\rangle)\leq 1\text{.}
\end{equation}
Finally, by Egorov's Theorem, we may assume without loss of generality that $Z$ and $\eta$ have been chosen so that each function $\beta_s$ is uniformly continuous on $Z$, with $|\beta(z_2)-\beta(z_1)|\leq\eta(|z_2-z_1|)$ for all $z_1,z_2\in Z$.

Fix $r>0$,
cover the compact set $Z$ with finitely many disjoint Borel subsets $C_1,\dotsc,C_m$, each of diameter at most $r$, and choose a point $c_j\in C_j$ for each $j$.  For each $s\in S$, we have
\begin{equation*}
g_s^1 = g_s^1(c_j) + \langle \diff{\pi} g^1_{s, c_j},\pi-\pi(c_j)\rangle + E^{g^1_s}_{c_j}|_{B_r(c_j)} = C+ \langle \diff{\pi} g^1_{s, c_j},\pi\rangle + E^{g^1_s}_{c_j}|_{B_r(c_j)}
\end{equation*}
for some constant C.

Then by equation \eqref{infinitesimal} and the locality axiom, we have
\begin{equation*}
\begin{aligned}
T(\omega)
&= \sum_{j=1}^m \sum_s T\restrict{C_j}(\beta_s \,dg_s^1 \wedge \dotsb\wedge dg_s^k)\\
&= \sum_{j=1}^m \sum_s T\restrict{C_j}\left(\beta_s \,d\left(\langle \diff{\pi} g^1_{s, c_j},\pi\rangle + E^{g^1_s}_{c_j}|_{B_r(c_j)}\right) \wedge dg_s^2 \wedge \dotsb\wedge dg_s^k\right)\text{.}
\end{aligned}
\end{equation*}
Therefore, since for each $j$, $L(E^{g^1_s}_z|_{B_r(c_j)})<\eta(r)$,  we have
\begin{align*}
&\left|T(\omega) - \sum_{j=1}^m \sum_s T\restrict{C_j}\left(\beta_s \,d(\langle \diff{\pi} g^1_{s, c_j},\pi\rangle) \wedge dg_s^2 \wedge \dotsb\wedge dg_s^k\right)\right|\\
\leq\, &\left|\sum_{j=1}^m \sum_s T\restrict{C_j}\left(\beta_s \,d(E^{g^1_s}_{c_j}|_{B_r(c_j)}) \wedge dg_s^2 \wedge \dotsb\wedge dg_s^k\right)\right|\\
\leq \,&\eta(r)\cdot\#S\cdot\sum_{j=1}^m ||T||(C_j) = \eta(r)\cdot\#S\cdot||T||(Z)\text{.}
\end{align*}
Arguing similarly for $i=2,\dotsc,k$, and additionally using inequality \eqref{bracketbound}, we have
\begin{align}
\label{repestimate}
&\left|T(\omega) - \sum_{j=1}^m \sum_s T\restrict{C_j}\left(\beta_s \,d(\langle \diff{\pi} g^1_{s, c_j},\pi\rangle) \wedge \dotsb\wedge d(\langle \diff{\pi} g^k_{s, c_j},\pi\rangle)\right)\right|\\
\leq &\,k\eta(r)\cdot\#S\cdot||T||(Z)\text{.}\notag
\end{align}
Moreover, since $|\beta_s(c)-\beta_s(c_j)|\leq\eta(r)$ for all $c\in C_j$, we can invoke the mass inequality \eqref{masscriterion1meas} to conclude that
\begin{equation}
\label{repmassestimate}
\left|T\restrict{C_j}\left((\beta_s-\beta_s(c_j))\,d(\langle \diff{\pi} g^1_{s, c_j},\pi\rangle) \wedge \dotsb\wedge d(\langle \diff{\pi} g^k_{s, c_j},\pi\rangle)\right)\right|\leq \eta(r)(1+\eta(r))^k||T||(C_j)\text{.}
\end{equation}
Combining inequalities \eqref{repestimate} and \eqref{repmassestimate} yields
\begin{align}
\label{repestimate1}
&\left|T(\omega) - \sum_{j=1}^m \sum_s T\restrict{C_j}\left(\beta_s(c_j) \,d(\langle \diff{\pi} g^1_{s, c_j},\pi\rangle) \wedge \dotsb\wedge d(\langle \diff{\pi} g^k_{s, c_j},\pi\rangle)\right)\right|\\
&\leq \#S\eta(r)(k+(1+\eta(r)^k))||T||(Z)\text{.}\notag
\end{align}
We next claim that
for each $j=1,\dotsc,m$, 
\begin{equation}
\label{multilinear}
\sum_{s\in S} T\restrict{C_j}\left(\beta_s(c_j) \,d(\langle \diff{\pi} g^1_{s, c_j},\pi\rangle) \wedge \dotsb\wedge d(\langle \diff{\pi} g^k_{s, c_j},\pi\rangle)\right)=0\text{.}
\end{equation}
Indeed,
equation \eqref{multilinear} may be rewritten
\begin{equation}
\label{multilinear1}
\sum_{s\in S} \beta_s(c_j)F_j((\diff{\pi} g^1_{s, c_j}, \dotsc, \diff{\pi} g^k_{s, c_j})) = 0\text{,}
\end{equation}
where $F_j\colon (\BBE^*)^k\rightarrow \BBR$ is given by
\begin{equation*}
F_j(\theta_1,\dotsc,\theta_k) = T\restrict{C_j}\left(d(\langle \theta_1,\pi\rangle) \wedge \dotsb\wedge d(\langle \theta_k,\pi\rangle)\right)\text{.}
\end{equation*}
By the linearity and alternating properties of currents, $F_j$ is linear and alternating, and therefore induces a linear map $\tilde{F}_j\colon \bigwedge^k\BBE^*\rightarrow \BBR$ such that $\tilde{F}_j(\theta_1\wedge\dotsb\wedge\theta_k)=F_j(\theta_1,\dotsb,\theta_k)$ for all $(\theta_1,\dotsb,\theta_k)\in (\BBE^*)^k$.  Therefore, we have
\begin{align*}
\sum_{s\in S} \beta_s(c_j)F_j(\diff{\pi} g^1_{s, c_j}, \dotsc, \diff{\pi} g^k_{s, c_j}) 
&= \sum_{s\in S} \tilde{F}_j(\beta_s(c_j)\diff{\pi} g^1_{s, c_j}\wedge \dotsb \wedge \diff{\pi} g^k_{s, c_j})\\
&= \tilde{F}\left(\sum_{s\in S}\beta_s(c_j)\diff{\pi} g^1_{s, c_j}\wedge \dotsb \wedge \diff{\pi} g^k_{s, c_j}\right)
= 0\text{,}
\end{align*}
since by assumption the argument in the last expression vanishes, and so the claim is proved.

Combining  equation \eqref{multilinear} with inequality \eqref{repestimate1}, we see that
\begin{equation*}
|T(\omega)|\leq \#S\eta(r)(k+(1+\eta(r)^k))||T||(Z)\text{.}
\end{equation*}
Passing to the limit as $r$ approaches $0$ completes the proof.
\end{proof}


Theorem \ref{representation} gives us an immediate bound on the dimension of most currents.
\begin{corollary}
\label{fewcurrents}
Suppose the chart $\pi\colon Y\rightarrow \BBE$ has dimension $n$, i.e., $\dim(\BBE)=n$.  Then there is a subset $Y_0\subset Y$, with $\mu(Y\backslash Y_0)=0$, such that every nonzero current concentrated on $Y_0$ has dimension at most $n$.
\end{corollary}
\begin{proof}
Let $\mathcal G$  be a countable dense subset of $\liploc(X)$.  Recall from Section \ref{preliminaries} that such a subset exists.  Since $\mathcal G$ is countable, the set $Y_0= \bigcap_{g\in\mathcal G} Y_g$ has full measure in $Y$.  On the other hand, for $k>n$, $\bigwedge^k \BBE^*=0$, so Proposition \ref{representationtheorem} implies that every $k$-current $T$ concentrated on $Y_0$ must satisfy 
\begin{equation*}
T(f \,dg^1\wedge \dotsb \wedge dg^k)=0
\end{equation*}
whenever each $g^i \in \mathcal G$.  The density of $\mathcal G$ in $\liploc(X)$ then implies that $T=0$.
\end{proof}

Though Theorem \ref{representation} itself is entirely coordinate free,  there are a number of consequences when coordinate functions are chosen for the differentiable structure.  Let $e_1,\dotsc,e_n$ be a basis for $\BBE$, and let $x^1,\dotsc,x^n\in \BBE^*$ be the corresponding dual basis. Also, let $\pi^i=x^i\circ \pi$. For every $g\in\liploc(X)$ and each $y\in Y_g$, let $\frac{\partial g}{\partial \pi^i}(y)  = \langle \diff{\pi} g_y,e_i\rangle$, so that 
\begin{equation}
\label{basicexpansion}
\diff{\pi}g_y=\sum_{i=1}^n \frac{\partial g}{\partial \pi^i}(y)\diff{\pi} \pi^i_y\text{.}  
\end{equation}

\begin{corollary}
\label{repk}
Let $(Y,\pi)$ be a coordinate chart on $X$,  let $\beta\,dg^1\wedge\dotsb\wedge dg^k\in\measextformsc{k}(X)$, and let $Y_{\mathcal G}=\cap_{i=1}^k Y_{g_i}$. Then for any current $T\in \currents{k}(X)$ such that $T\restrict{\beta}$ is concentrated on $Y_{\mathcal G}$,  
\begin{equation}
\label{repkeq}
T(\beta \,dg^1\wedge\dotsb\wedge dg^k)
=  T\left(\sum_{a \in \Lambda_{k,n} }\beta\det\left(\frac{\partial g^i}{\partial \pi^j}\right) \,d \pi^{a_1}\wedge \dotsb \wedge d\pi^{a_k} \right)\text{.}
\end{equation}
\end{corollary}
\begin{proof}
The corresponding differential forms for both sides are equal when defined, i.e.,
\begin{equation*}
\beta\,\diff{\pi}g^1\wedge\dotsb\wedge \diff{\pi}g^k - \beta \sum_{a \in \Lambda_{k,n} }\left(\det\left(\frac{\partial g^i}{\partial \pi^j}\right) \,\diff{\pi} \pi^{a_1}\wedge \dotsb \wedge \diff{\pi} \pi^{a_k} \right) = 0
\end{equation*}
almost everywhere.  Applying Theorem \ref{representation} then completes the proof.
\end{proof}

\begin{proof}[Proof of Theorem \eqref{cheegerchainrule}.]
Apply Corollary \ref{repk} to the current $F_\#(T\restrict{\beta})$ and the form $dg^1\wedge\dotsb\wedge dg^k$.  
\end{proof}

We are about ready to prove Theorem \ref{vectormeasure}, but first we must define precurrents precisely.

\begin{definition}
\label{weakdef}
A linear map $T\colon\measformsc{k}(X)\rightarrow \BBR$ is a \textbf{$k$-precurrent on $Y$} if 
\begin{equation*}
T( \beta \,d g^1 \wedge \dotsb \wedge d g^m)  =\int_Y \langle  \beta\,\diff{\pi} g^1 \wedge \dotsb \wedge \diff{\pi} g^k,\hat{\lambda}\rangle \, d\mu\text{,}
\end{equation*}
for some locally integrable map $\hat{\lambda}\colon Y \rightarrow \bigwedge^k \BBE$. Such a map $\hat{\lambda}$ is called a \textbf{(measurable) $k$-vector field}.
If $T\restrict{Y}$ is a $k$-precurrent on $Y$ for every coordinate patch $Y$, we simply say that $T$ is a \textbf{$k$-precurrent}.  
\end{definition}
Note that the linearity and locality axioms from Definition \ref{currentdef} are easily seen to be satisfied, but the continuity axiom need not be, as Theorem \ref{characterization} demonstrates.

\begin{proof}[Proof of Theorem \ref{vectormeasure}]
We must show that $T\restrict Y$ is a precurrent for the chart $Y$, whenever $Y$ is a coordinate chart.  We fix such a chart $Y$.

The correspondence between $0$-currents and measures given in Lemma \ref{0currents} says that there are functions $\lambda^a \in L^\infty(Y,||T||)$ such that
\begin{equation}
\label{0currentapplication}
T(\beta\,d\pi^{a_1}\wedge\dotsb\wedge d\pi^{a_k}) =T\restrict{d\pi^{a_1}\wedge\dotsb\wedge d\pi^{a_k}}(\beta) = \int_Y \beta\lambda^a \,d\mu
\end{equation}
for any $\beta\in\boundedc(X)$.  Since $||T||\ll \mu$, the functions $\lambda^a$ are locally $\mu$-integrable.  Let $\hat{\lambda}\colon Y\rightarrow \bigwedge^k\BBE$ be given by
\begin{equation*}
\hat{\lambda}^\alpha = \sum_{a \in \Lambda_{k,n_{\alpha}}} \lambda^{\alpha,a}e_{a_1}\wedge\dotsb\wedge e_{a_k}\text{.}
\end{equation*} 
Since $Y_{\mathcal G}$ has full $\mu$-measure, and by assumption, $T$ is concentrated on $Y$ with $||T||\ll \mu$, we see that $T$ is concentrated on $Y_{\mathcal G}$.  Thus we may invoke Corollary \ref{repk}.  Applying equation \eqref{0currentapplication} to the right hand side of \eqref{repkeq}, we see that
\begin{align*}
T(f \,d g^1 \wedge \dotsb \wedge d g^k) 
=  \sum_{a \in \Lambda_{k,n}} T\left(f \det\left(\frac{\partial g^i}{\partial \pi^{a_j}}\right)\,d\pi^{a_1}\wedge\dotsb\wedge d\pi^{a_k}\right)\\
=  \sum_{a \in \Lambda_{k,n}} \int_Y f \lambda^a\det\left(\frac{\partial g^i}{\partial \pi^{a_j}}\right) \,d\mu
=\int_Y \langle  f\,\diff{\pi} g^1 \wedge \dotsb \wedge \diff{\pi} g^k,\hat{\lambda}\rangle \, d\mu
\text{.}
\end{align*}
\end{proof}



%% file: Metricgroups.tex
\section{Metric groups}

\label{metricgroups}

We begin our study of currents in Carnot groups with a more general setting.  Suppose $\Gamma = (\Gamma,\gunit,\dist(\cdot,\cdot),\mu)$ is a locally compact group with identity $\gunit$, left-invariant metric $\dist(\cdot,\cdot)$ and left Haar measure $\mu$.  In such a group, we will abuse notation and identify an element $\gamma \in \Gamma$ with the associated left translation map $\alpha\mapsto \gamma\alpha$.

Our main result in this section is that on a metric group, the set of $k$-currents of absolutely continuous mass is weakly dense:
\begin{proposition}
\label{approximationlemma}
Let $T\in \currentsloc{k}(\Gamma)$ be a current of locally finite mass in a metric group $\Gamma$.  Then there are currents $T_\epsilon\in\currentsloc{k}(\Gamma)$ whose masses $||T_\epsilon||$ are absolutely continuous with respect to $\mu$, and such that $T_\epsilon$ converges weakly to $T$ as $\epsilon$ converges to $0$, i.e.,
\begin{equation}
\label{weakconv}
\lim_{\epsilon\rightarrow 0} T_\epsilon(\omega) = T(\omega)
\end{equation}
for each $\omega\in\formsc{k}(\Gamma)$.
\end{proposition}

\begin{proof}

For $\omega\in \formsc{k}(\Gamma)$, we define $$T_\epsilon(\omega) = \dashint_{B(\gunit,\epsilon)} (\gamma_\#T)(\omega)\,d\mu(\gamma)\text{.}$$

We must first check that for each $\epsilon>0$, $T_\epsilon$ is a current.  Fix $\epsilon$, and suppose the forms $\omega_i=f_i\,dg^1_i\wedge\dotsb\wedge dg^k_i$ converge to $\omega=f\,dg^1\wedge\dotsb\wedge dg^k$. 

Since the functions $f_i$ converge uniformly to $f$, and the translation maps are isometries, all of the functions $f_i\circ \gamma$ are uniformly bounded in absolute value, say 
\begin{equation}
\label{uniformabs}
|f_i\circ \gamma|\leq M\text{.}
\end{equation}
Similarly, the functions $g_i^j$ have locally uniformly bounded Lipschitz constants, and so there is some $N>0$ such that $L(g_i^j|_{N_{2\epsilon}(K)})<N$ for all $i$ and $j$, where $K = \bigcup \supp(f_i)$, and $N_{2\epsilon}(K)=\{\gamma\in\Gamma:\dist(\gamma,K)<2\epsilon\}$. 
It follows, again because the translation mappings are isometries, that for each $\gamma\in B_\epsilon(\gunit)$, and each $i$ and $j$,
\begin{equation}
\label{uniformlip}
L(g_i^j\circ\gamma|_{N_{\epsilon}(K)})<N\text{.}
\end{equation}

Inequalities \eqref{uniformabs} and \eqref{uniformlip}, as well as the mass criterion \eqref{masscriterion1}, imply that for all $i$ and $j$, and for $\gamma\in B_\epsilon(\gunit)$,
\begin{equation*}
\gamma_\#T(\omega_i)\leq MN^k||T||(N_{\epsilon}(K))\text{.}
\end{equation*}
Moreover, by the continuity axiom, for each $\gamma$, $T(\omega_i)$ converges to $T(\omega)$.  Thus $T_\epsilon(\omega_i)$ converges to $T_\epsilon(\omega)$ by the Lebesgue Dominated Convergence Theorem.

To prove that $||T_\epsilon|| \ll \mu$, it is enough to show that whenever $\mu(A) = 0$, $T_\epsilon\restrict{A} = 0$, since this implies $||T_\epsilon||\restrict{A} = ||T_\epsilon\restrict{A}||=0$.  To establish that $T_\epsilon\restrict{A} = 0$, we argue as follows: If $A\subset \Gamma$ with $\mu(A)=0$, we use Fubini's theorem and Definition \ref{massdef} to conclude that
\begin{align*}
&|T_\epsilon\restrict{A}(f\,dg^1\wedge \dotsb \wedge dg^k)| 
= \left|\dashint_{B(\gunit,\epsilon)} T\left((\charfcn_A f)\circ \gamma \,d(g^1\circ \gamma) \wedge \dotsb \wedge d(g^k\circ \gamma)\right)\,d\mu(\gamma)\right|\\
&\leq \dashint_{B(\gunit,\epsilon)} \left|T\left((\charfcn_A f)\circ \gamma \,d(g^1\circ \gamma) \wedge \dotsb \wedge d(g^k\circ \gamma)\right)\right|\,d\mu(\gamma)\\
&\leq N^k\sup_{|f|\leq 1} \left(\dashint_{B(\gunit,\epsilon)} \left(\int_\Gamma|(\charfcn_A f)(\gamma y)| \,d||T||(y)\right)\,d\mu(\gamma)\right)\\
&=N^k\sup_{|f|\leq 1} \left(\int_\Gamma\left(\dashint_{B(\gunit,\epsilon)} |(\charfcn_A f)(\gamma y)| \,d\mu(\gamma)\right)\,d||T||(y)\right)\\
&=N^k\sup_{|f|\leq 1} \left(\int_\Gamma\left(\dashint_{B(\gunit,\epsilon)} |(\charfcn_{Ay}(\gamma) f(\gamma y)| \,d\mu(\gamma)\right)\,d||T||(y)\right)
=0\text{.}
\end{align*}

Note that the second to last line vanishes because right translations map null sets to null sets.  This follows from the fact that left and right Haar measure are in the same measure class, and thus have the same null sets, so that $\mu(A_y)=0$ for all $y\in \Gamma$.

It now remains only to check that (\ref{weakconv}) holds for every $\omega\in\formsc{k}(\Gamma)$.

We argue by contradiction. Suppose $T_\epsilon$ does not converge weakly to $T$.  Then for some $\omega = f \,dg^1\wedge\dotsb\wedge dg^k \in \formsc{k}(\Gamma)$, $\delta>0$, and some sequence $\{\epsilon_i\}$ with $\epsilon_i\rightarrow 0$ as $i\rightarrow \infty$, we have 
$$|(T_{\epsilon_i}-T)(\omega)|\geq\delta\text{.}$$
For each $i$, we therefore have some $\gamma_i\in B(\gunit,\epsilon_i)$ such that 
\begin{align*}
\delta &\leq  |(\gamma_{i\#}T-T)(\omega)|\\
&\leq |T(f\circ \gamma_i\,d(g^1\circ \gamma_i)\wedge\dotsb\wedge d(g^k\circ \gamma_i))-T(f\,dg^1\wedge\dotsb\wedge dg^k)|\text{.}
\end{align*}

On the other hand, $\gamma_i\rightarrow \gunit$, from which it follows that $f\circ \gamma_i$ converges to $f$ in $\lipc(\Gamma)$, and $g^j\circ \gamma_i$ converges to  $g^j$ in $\liploc(\Gamma)$ for each $j$.  
This contradicts the continuity of $T$.
\end{proof}

Proposition \ref{approximationlemma}, in combination with Corollary \ref{fewcurrents} and the alternating property, immediately yields the following result.
\begin{corollary}
\label{nogroupcurrents}
Let $\Gamma$ be a metric group with a differentiable structure of dimension $n$.  Then $\Gamma$ admits no nonzero $k$-currents for $k>n$.
\end{corollary}

%% file: Carnotgroups.tex
\section{Carnot groups.}
\label{carnot}
We recall some definitions and facts about stratified Lie groups, also known as Carnot groups, equipped with their Carnot-Carath\'eodory metrics.  All of this material is surveyed in \cite{HeinonenCarnot}.  A much more in-depth study of Carnot-Carath\'eodory spaces can be found in \cite{Gromov}, and of Carnot groups specifically, in \cite{Pansu}.
\begin{definition}
A \textbf{Carnot group} is a connected, simply connected Lie group $\BBG=(\BBG,\gtimes)$, with unit element $\gunit$ and left Haar measure $\mu$, whose Lie algebra $\mathfrak g = T_\gunit\BBG$, with bracket $[\cdot,\cdot]$, admits a stratification, i.e., a direct sum decomposition
\begin{equation*}
\mathfrak g = V_1 \oplus \dotsb \oplus V_m
\end{equation*}
such that $[V_1,V_j]= V_{j+1}$ for $j<m$, and $[\mathfrak g,V_m]= 0$.
\end{definition}
We call $\BBG$ a Carnot group of step $m$.

For $p\in\BBG$, let $\tau_p$ denote the left-translation map $q\mapsto p\gtimes q$. 
Throughout this chapter, we will take the point of view that $k$-vector fields and $k$-forms, respectively, are maps from $\BBG$ into $\bigwedge^k\mathfrak g$ and $\bigwedge^k \mathfrak g^*$. 
Notice that this agrees with the usual notion by way of the canonical identification between $T_p$ and $\mathfrak g = T_{\gunit}$ given by the translation map $\tau_{p*}$.

We assume $\mathfrak g$ is equipped with an inner product $\langle\cdot,\cdot\rangle$, so that $\BBG$ has a left-invariant Riemannian structure.  We denote by $\rdist(\cdot,\cdot)$ the metric induced by this structure.  

We refer to $H= V_1$ as the \textit{horizontal subspace}. 
The vector bundle $\mathcal H = \bigcup_{p\in \BBG} \tau_{p*} H$ is called the \textit{horizontal bundle}.
A piecewise smooth path $\gamma\colon I\rightarrow \BBG$ is said to be \textbf{horizontal} if $\frac{d\gamma}{dt}\in \mathcal H$ for all but finitely many $t\in I$.

\begin{definition}
The Carnot-Carath\'eodory distance between two points $p,q\in \BBG$ is
\begin{equation*}
\cdist(p,q) = \inf \{l(\gamma):\text{$\gamma$ is a horizontal path joining $p$ and $q$.}\}
\end{equation*}
\end{definition}
It is a deep result of Chow \cite{Chow} and Rashevsky \cite{Rashevsky} that the Carnot-Carath\'eodory distance is in fact finite, and therefore a metric on $\BBG$.

If $v\in \mathfrak g$, we denote by $X^v$ the unique left invariant vector field on $\BBG$ satisfying $X^v_\gunit=v$.  

Finally, if $f\colon \BBG\rightarrow \BBR$ is differentiable (in the usual sense, as opposed to the Pansu-differentiability described below) at $p\in\BBG$, we write $\driem f_p\colon \mathfrak g\rightarrow \BBR$ for the differential of $f$, as the symbol $df$ has already been expropriated for metric currents.  The ``$r$'' is to emphasize that this differential is the one that should exist almost everywhere (by Rademacher's theorem) for functions that are Lipschitz in the \textit{Riemannian} metric on $\BBG$.  A theorem of Pansu (Theorem \ref{pansu} below) provides an analogous differential, $\dpansu$, for Lipschitz functions in the Carnot-Carath\'eodory metric.



A Carnot group's Lie algebra $\mathfrak g$ is equipped with a one-parameter family of linear dilations $\delta_r\colon \mathfrak g \rightarrow \mathfrak g$ given by $\delta_r(v_j)=r^j v_j$ for $v_j \in V_j$.  The maps $\delta_r$ are Lie algebra homomorphisms, and so induce Lie group homomorphisms $\Delta_r\colon\BBG \rightarrow \BBG$ via the exponential map, such that the $\Delta_{r*}(\gunit)= \delta_r$.  
Notice that since $\Delta_r$ is a homomorphism, we have $\Delta_r\circ \tau_p = \tau_{\Delta(p)} \circ \Delta_r$ for every $p\in \BBG$.  It follows that for every $u\in H$, $p\in \BBG$, 
and $r>0$, we have
\begin{equation}
\label{dilation}
\Delta_{r*} X^u_p 
= \Delta_{r*} \tau_{p*} u 
= \tau_{\Delta_r(p)*}\Delta_{r*} u 
= r \tau_{\Delta_r(p)*} u 
= r X^u_{\Delta_r(p)}\text{.}
\end{equation}

Thus the dilation $\Delta_r$ rescales the metric $\cdist$ by a factor of $r$, as the name implies.

The number $Q = \sum_{i=1}^{m} i \dim(V_i)$ is called the \textit{homogeneous dimension} of $\BBG$.  As motivation, we note that the dilations $\Delta_r$ have Jacobian $r^Q$.  
A Carnot $\BBG$ with homogeneous dimension $Q$ has Hausdorff dimension $Q$ as well, and is in fact Ahlfors $Q$-regular \cite{HeinonenCarnot}.  Since the metric $\cdist$ is invariant under left translations, and the Hausdorff $Q$-measure $\mathcal H^Q$ is positive and finite on balls, we adopt the convention that the Haar measure $\mu=\mathcal H^Q$. Note that this implies
\begin{equation}
\label{pushforwardcarnot}
\Delta_{r\#}\mu = r^{-Q}\mu
\end{equation}
for each $r>0$. For a noncommutative Carnot group (i.e., one of step $m>1$), $Q$ always exceeds the topological dimension, and so such groups give us a rich supply of fractal spaces to study.

Lastly, we note that Carnot groups, being nilpotent, are unimodular \cite{ReiterStegeman}.

\begin{example}
\label{heisex}
The $n^{\text{th}}$ Heisenberg group $\heis{n}$ is a $(2n+1)$-dimensional Lie group whose Lie algebra is spanned by vector fields $X_i$, $Y_i$ and $Z$, for $i=1, \dotsc, n$, satisfying the relations 
$$[X_i,Y_i] = Z$$
with all other generators commuting.  
The group $\heis{n}$ is a step-$2$ Carnot group with stratification $\spa(X_1,Y_1,\dotsc,X_n,Y_n) \oplus \spa(Z)$.  The homogeneous dimension $Q$ is $2n+2$, one more than the topological dimension.
\end{example}

\subsection*{Density of smooth functions}
The following lemma will allow us to employ the smooth structure of a Carnot group $\BBG$ in our analysis of $\currents{k}(\BBG)$.  

\begin{lemma}
\label{smoothdense}
The space $\smoothc(\BBG)$ of smooth functions on $\BBG$ with compact support is dense in $\lipc(\BBG)$.  Similarly, $\smooth(\BBG)$ is a dense subset of $\liploc(\BBG)$.
\end{lemma}
\begin{proof}
The proof is a standard smoothing argument, and is essentially the same as the argument given for the case $\BBG=\BBR^n$ in \cite[Section 1.5]{Lang}.

Let $f\in\lipc(\BBG)$, with Lipschitz constant $L$.
Let $\phi\colon \BBG\rightarrow [0,\infty)$ be a smooth function supported on $B_1(\gunit)$ such that $\int_\BBG \phi\, d\mu= 1$. 
For every $\epsilon>0$, define $\phi_\epsilon(p)= \epsilon^{-Q}\phi\circ\Delta_{\epsilon}$.
Note that $\phi_\epsilon$ is supported on $B_\epsilon(\gunit)$, and that $\int_\BBG \phi_\epsilon = 1$.
We then define smooth functions $f_\epsilon\colon \BBG\rightarrow \BBR$ by
\begin{equation*}
f_\epsilon(p)=\int_\BBG f(q^{-1}p)\phi_\epsilon(q) d\mu(q) = \int_\BBG f(z)\phi_\epsilon(pz) d\mu(z)\text{.}
\end{equation*}
Then at every $p\in\BBG$, and for every $\epsilon>0$,
\begin{equation*}
|f_\epsilon(p) - f(p)|\leq \int_\BBG |f(z)\phi_\epsilon(pz) -f(p)| d\mu(z)\text{.}
\end{equation*}
By continuity of $\phi$, the right hand side converges to $0$ with $\epsilon$, so that $f_\epsilon$ converges pointwise to $f$.
Moreover, for every $p_1,p_2\in \BBG$, we have
\begin{align*}
&|f_\epsilon(p_1)-f_\epsilon(p_2)| = \left|\int_\BBG \left(f(q^{-1}p_1)-f(q^{-1}p_2)\right)\phi_\epsilon(q) d\mu(q)\right|\\
&\leq \int_\BBG \left|f(q^{-1}p_1)-f(q^{-1}p_2)\right|\phi_\epsilon(q) d\mu(q)
\leq \int_\BBG \left|L\cdist(q^{-1}p_1,q^{-1}p_2)\right| \phi_\epsilon(q) d\mu(q)\\
&\leq \int_\BBG \left|L\cdist(p_1,p_2)\right| \phi_\epsilon(q) d\mu(q)
= L \cdist(p_1,p_2)\text{.}
\end{align*}
Thus the functions $f_\epsilon$ have uniformly bounded Lipschitz constant.  Moreover, for $\epsilon < 1$, they are supported on the relatively compact neighborhood $\mathcal N_1(\supp(f))=\{p\in\BBG:\cdist(p,\supp(f))<1\}$.  Therefore $f_\epsilon$ converges in $\lipc(\BBG)$ to $f$.

The same argument shows the density of $\smooth(\BBG)$ in $\liploc(\BBG)$.  The only different part of the argument is to show that the functions are locally uniformly Lipschitz.  To see this, note that for any compact set $K\subset \BBG$, if $f|_{\mathcal N_1(K)}$ is $L$-Lipschitz, then $f_\epsilon|_K$ is $L$-Lipschitz for $\epsilon<1$, and so the Lipschitz constants of $f_\epsilon|_K$ are uniformly bounded for each $K$. 
\end{proof}

\subsection*{Differentiable structure}
According to a result of Jerison \cite[Theorem 2.1]{Jerison}, a Carnot group admits a Poincar\'e inequality, and thus by Cheeger's differentiation theorem, also admits a differentiable structure.  In  fact, the structure can be described by differentiating in the horizontal directions, as stated precisely in Theorems \ref{pansu} and \ref{pansucheeger} below, due to Pansu and Cheeger-Weaver, respectively.

Before we state the theorem, a number of remarks are in order.  First, the Lie subalgebra $\mathfrak v = 0\oplus V_2 \oplus\dotsb\oplus V_m$ is an ideal, and so the corresponding Lie subgroup $\BBV \subset \BBG$ is normal \cite{Hall}. Moreover, we can identify $H$ with $\mathfrak g/\mathfrak v$ by way of the quotient map $\pi_*\colon  \mathfrak g\rightarrow \mathfrak g/\mathfrak v$  (here $\pi\colon \BBG \rightarrow \BBG/\BBV$ is the quotient map between Lie groups).  By way of this identification, we equip $\mathfrak g/\mathfrak v$ with the inner product from $H$, and notice that with respect to this inner product, the map $\pi_*$ is $1$- Lipschitz. It follows that the map $\pi$ is Lipschitz with respect to the Carnot Carath\'eodory metric on $\BBG$ and the Riemannian metric on $\BBG/\BBV$ (which is just a Euclidean metric).  In the future, we will denote by $\BBH$ the group $\BBG/\BBV$, equipped with the aforementioned metric. We will also denote by $\mathfrak h = \mathfrak g/\mathfrak v$ the Lie algebra of the group $\BBH$.


The following generalization of Rademacher's differentiation theorem was proved by Pansu \cite{Pansu}.
\begin{theorem}[{\cite[Theorem 2]{Pansu}}]
\label{pansu}
Let $f\colon \BBG_1\rightarrow \BBG_2$ be a Lipschitz mapping between two Carnot groups.  For every $p\in \BBG$ and $t>0$, define $f^t_p\colon \BBG_1\rightarrow \BBG_2$ to be the rescaling 
\begin{equation}
\label{rescalings}
f^t_p(q)  = \Delta_t^{-1}(f(p)^{-1}\cdot f(p\cdot\Delta_t(q)))\text{.}
\end{equation}
Then at almost every $p\in \BBG_1$, there is a Lie group homomorphism $\Dpansu f_p\colon \BBG_1\rightarrow \BBG_2$, commuting with each dilation map $\Delta_t$, 
given by 
\begin{equation}
\label{dpansudef}
\Dpansu f_p(q) = \lim_{t\rightarrow 0} f^t_p(q)
\end{equation}
\end{theorem}

We call $\Dpansu f_p$ the \textit{Pansu differential} at $p$.  When it is defined, we say $f$ is \textit{Pansu-differentiable} at $p$.  Notice that 
each of the maps $f^t_p$ are Lipschitz with the same Lipschitz constant $L(f)$, so that if it exists, $\Dpansu f_p$ is $L(f)$-Lipschitz as well. We also define $\dpansu f_p$ to be the induced Lie algebra homomorphism $\dpansu f_p = \left(\Dpansu f_{p}\right)_*$.

We are interested in the case where $\BBG_2=\BBR$.  In this case, since $\BBR$ is Abelian, the map $\Dpansu f_p$ vanishes on $\BBV$, since the latter group is the commutator of $\BBG$, as follows from the stratification of $\mathfrak g$.  Therefore there is an induced homomorphism $\Dpansuh f_p\colon \BBH \rightarrow \BBR$ such that $\Dpansuh f_p\circ \pi = \Dpansu f_p$.  Note then that $\dpansuh f_{p}\circ \pi_* = \dpansu f_{p}\colon \mathfrak g\rightarrow \BBR$.  Also, since $\dpansu f_{p}$ is an element of $\mathfrak g^*$, we will write $\dpansu f_{p}(u) = \langle \dpansu f_{p},u\rangle$. 

The stratification of $\BBG$ indicates that $\exp(tu)=\exp(\delta_t u) = \Delta_t (\exp(u))$ for every $u\in H$.
It follows that at every point $p\in \BBG$ of Pansu differentiability, and for every $u\in H$, the partial derivatives $X^u_p(f)$ exist, and we have
\begin{align}
X^u_p(f)
= \frac{d}{dt}|_{t=0} f(p\cdot \exp(tu))
= \lim_{t\rightarrow 0} \frac{f(p\cdot \exp(tu))-f(p)}{t}\label{partialpansu}\\
= \lim_{t\rightarrow 0} \frac{f(p\cdot \Delta_t\exp(u))-f(p)}{t}
= \Dpansu f_p(\exp(u))
= \langle \dpansu f_{p},u\rangle\text{.}\notag
\end{align}
Moreover, if $f\colon \BBG\rightarrow \BBR$ is differentiable (in the usual sense), then for any $q\in \BBG$, the map $t\mapsto f(p\cdot\Delta_t(q))$ is differentiable at $t=0$, from which it follows that $f$ is Pansu differentiable at $p$.  From equation \eqref{partialpansu}, then, we have that 
\begin{equation}
\label{pansusmooth}
\langle \dpansu f_p,u\rangle = X^u_p(f) = \langle u, \driem f_p\rangle\text{.}
\end{equation}

Note that the Pansu differential is compatible with dilations in the following sense.  If $f\colon \BBG \rightarrow \BBR$ is a Lipschitz function, and $r>0$, then 
\begin{align*}
&\Dpansu (f\circ \Delta_r)_p (q) 
=\lim_{t\rightarrow 0} (f\circ \Delta_r)^t_p(q)
=\lim_{t\rightarrow 0} \frac{f(\Delta_r(p \Delta_t (q)))-f(\Delta_r(p))}{t}\\
&= \lim_{t\rightarrow 0} r\cdot \frac{f(\Delta_r(p) \Delta_{rt}(q))-f(\Delta_r(p))}{rt}
=r \Dpansu f_{\Delta_r(p)}(q)\text{.}
\end{align*}
Differentiating, we obtain
\begin{equation}
\label{pansudilation}
\dpansu (f\circ\Delta_r)_p = r \dpansu f_{\Delta_r(p)}\text{.}
\end{equation}

By a theorem of Cheeger and Weaver, differentiation in the horizontal directions provides a concrete description of the differential structure of a Carnot group. 

\begin{theorem}[{\cite[Remark 4.66]{Cheeger}, \cite[Theorems 39 and 43]{Weaver}}]
\label{pansucheeger}
Let $\BBG$ be a Carnot group with $H$, $\mathfrak h$, and $\BBH$ as defined above.  Then $\BBG$ admits a differentiable structure with a single coordinate chart, namely the quotient map $\pi\colon \BBG\rightarrow \BBH$ defined above.  For every $f\in\lip(\BBG)$, the differential $\diff{\pi} f\colon \BBG \rightarrow \mathfrak h^*$ is given by $\diff{\pi} f_p = \dpansuh f_p$, whenever the latter is defined. 
If $p$ is a point of (Pansu) differentiability, then for every $u\in H$, $X^u_p(f)$ exists and satisfies 
\begin{equation}
\label{partialcheeger}
\langle \diff{\pi} f_p,\pi_*u\rangle = X^u_p(f)\text{.}
\end{equation} 
\end{theorem}

%% file: Precurrentsincarnotgroups.tex
\section{Precurrents in Carnot groups.}
\label{carnotprecurrents}

From Theorem \ref{pansucheeger}, we know that precurrents in Carnot groups have the form
\begin{equation}
T( f \,d g^1 \wedge \dotsb \wedge d g^k) =   \int_\BBG \langle  f \, \dpansuh g^1\wedge \dotsb\wedge \dpansuh g^k,\hat{\lambda}\rangle \, d\mu\text{,}
\end{equation}
where $\hat{\lambda}\colon \BBG\rightarrow \bigwedge^k \mathfrak h$  is locally integrable.  Since the restriction of the projection map $\pi_*|_{\bigwedge^k H}\colon \bigwedge^k H\rightarrow \bigwedge^k \mathfrak{h}$ is an isomorphism (via the isomorphism $\pi_*|_H\colon H\rightarrow \mathfrak{h}$), it follows that there is a locally integrable $k$-vector field $\tilde{\lambda} \colon \BBG \rightarrow \bigwedge^k H$ such that
\begin{align}
T(f \,d g^1 \wedge \dotsb \wedge d g^k)
=   \int_\BBG \langle f(p) \, \dpansuh g^1_{p}\wedge \dotsb\wedge \dpansuh g^k_{p},\pi_*(\tilde{\lambda}_p)\rangle \, d\mu(p)\label{weakcarnot}\\
=   \int_\BBG \langle f(p) \, \dpansu g^1_{p}\wedge \dotsb\wedge \dpansu g^k_{p},\tilde{\lambda}_p\rangle \, d\mu(p)\notag
\end{align}
We denote the precurrent in the above equation by $T_{\tilde{\lambda}}$.  For the rest of this chapter, all $k$-vector fields under consideration will be locally integrable, and so we generally omit this modifier and simply refer to such an object as a $k$-vector field.

Let $\pi^1,\dotsc,\pi^n$ be defined as in the discussion following Corollary \ref{fewcurrents}.  Let $\{u_1,\dotsc,u_n\}$ be the (orthonormal) basis for $H$ dual to $\{\dpansu \pi^1\,\dotsc,\dpansu \pi^n\}$.  Then the simple $k$-vectors $\tilde{u}_a=u_{a_1} \wedge\dotsb\wedge u_{a_k}$ form a basis for $\bigwedge^k H$, and so every $k$-vector field $\tilde{\lambda}$ has the form
\begin{equation}
\label{kvectexp}
\tilde{\lambda} = \sum_{a \in \Lambda_{k, n}} \lambda^a \tilde{u}_a\text{,}
\end{equation}
for locally integrable functions $\lambda^a$ on $\BBG$.

An initial observation is that, as with the currents described in this paper, precurrents have locally finite mass.
\begin{lemma}
\label{prefinitemass}
Let $T$ be a $k$-precurrent in $\BBG$.  Then $T$ has locally finite mass.
\end{lemma}
\begin{proof}
Let $T=T_{\tilde{\lambda}}$, let $\omega \in\formsc{k}(U)$ with $||\omega||\leq 1$, and let $U\subset \BBG$, with $\overline{U}$ compact. We may then write $\omega = \sum_{s \in S} f_s\,dg^1_s\wedge\dotsb\wedge dg^k_s$, with $\sum_{s\in S}|f_s|\leq (1+\epsilon)$ for some $\epsilon>0$, and $g^i_s\in\lip_1(U)$ for each $i$ and $s$.  Since $||u_j||=1$, equation \eqref{pansusmooth} implies that $\left|\langle  (\dpansu g^i_s)_p,u_j\rangle\right| = |X^u_p(g^i_s)|\leq 1$ for every $i$ and $j$, so that
\begin{equation*}
|\langle  f \dpansu g^1_s \wedge\dotsb \wedge \dpansu g^k_s,\tilde{u}_a\rangle| \leq \frac{n!}{k!} |f|\text{.}
\end{equation*}
We therefore compute
\begin{align}
&|T_{\tilde{\lambda}}(\omega)| 
\leq \sum_{s \in S} |T(f_s\,dg^1_s\wedge\dotsb\wedge dg^k_s)|
\leq \sum_{s\in S} \sum_{a\in\Lambda(k,n)} |T_{\lambda^a \tilde{u}_a} (f_s\,dg^1_s\wedge\dotsb\wedge dg^k_s)\notag\\
&\leq \sum_{s\in S} \sum_{a\in\Lambda(k,n)} \int_U  |\lambda^a|\cdot|\langle  f_s \dpansu g^1_s \wedge\dotsb \wedge \dpansu g^k_s,\tilde{u}_a\rangle| \,d\mu\label{massestimate}\\
&\leq \sum_{s\in S} \sum_{a\in\Lambda(k,n)} \int_U  |\lambda^a| \frac{n!}{k!}|f_s|\,d\mu
\leq (1+\epsilon)\frac{n!}{k!}  \sum_{a\in\Lambda(k,n)} \int_U  |\lambda^a|\,d\mu\text{,}\notag
\end{align}
where the last line is finite as a result of the local integrability of the functions $\lambda^a$.
\end{proof}

As is the case with currents, the finite mass condition extends the domain of a precurrent to $\measformsc{k}(\BBG)$, and to $\measextformsc{k}(\BBG)$.

We define restrictions of precurrents exactly as we did in Section \ref{currentdefinitions} for currents.  That is, if $T$ is a $k$-precurrent, and $\omega\in\measforms{j}(\BBG)$, we define the restriction
$T\restrict{\omega}$ by equation \eqref{restrictioneq}.

Recall that given a $k$-vector $\hat{u}\in \bigwedge^k \mathfrak g$ and a $j$-covector $\hat{a}\in \bigwedge^j \mathfrak g^*$, there is a unique $k-j$ vector $\hat{u}\restrict{\hat{a}}\in \bigwedge^{k-j} \mathfrak g$ such that for all $\hat{b}\in \bigwedge^{k-j} \mathfrak g^*$,
$\langle \hat{b},\hat{u}\restrict{\hat{a}}\rangle = \langle \hat{a}\wedge \hat{b},\hat{u} \rangle$.
Thus 
\begin{align*}
&T_{\tilde{\lambda}}\restrict{\beta\,dh^1\wedge\dotsb\wedge dh^j}(f\,dg^1\wedge\dotsb\wedge dg^{k-j})\\
&=\int_\BBG \langle  \beta f \, \dpansu h^1\wedge\dotsb\wedge \dpansu h^j \wedge \dpansu g^1\wedge \dotsb\wedge \dpansu g^k,\tilde{\lambda}\rangle \, d\mu\\
&=  \int_\BBG \langle  f\, \dpansu g^1\wedge \dotsb\wedge \dpansu g^k,\tilde{\lambda}\restrict{\beta \, \dpansu h^1\wedge\dotsb\wedge \dpansu h^j}\rangle \, d\mu\\
&= T_{\tilde{\lambda}\restrict{\beta \, \dpansu h^1\wedge\dotsb\wedge \dpansu h^j}}(f\,dg^1\wedge\dotsb\wedge dg^{k-j})\text{,}
\end{align*}
and so a restriction of a precurrent is again a precurrent.


As a result of the expansion in equation \eqref{kvectexp},  every precurrent $T$ has the form
\begin{equation}
\label{weakexpansion}
T = T_{\tilde{\lambda}} = \sum_{a \in \Lambda_{k, n}} T_{\lambda^a \tilde{u}_a} = \sum_{a \in \Lambda_{k, n}} T_{\tilde{u}_a}\restrict{\lambda^a}\text{.}
\end{equation}

\begin{remark}
\label{uniquevectorfield}
Note that $T=0$ if and only if $\lambda^a=0$ almost everywhere for each $a\in\Lambda_{k,n}$.  Indeed, if $\lambda^b\neq 0$ on a set of positive measure, then 
\begin{equation*}
T_{\tilde{\lambda}}\restrict{d\pi_{b_1}\wedge\dotsb\wedge d\pi_{b_k}}
= \sum_{a \in \Lambda_{k, n}} T_{\tilde{u}_a\restrict{\lambda^a \dpansu \pi_{b_1}\wedge\dotsb\wedge \dpansu \pi_{b_k}}}
=  T_{\tilde{u}_b\restrict{\lambda^b \dpansu \pi_{b_1}\wedge\dotsb\wedge \dpansu \pi_{b_k}}}
= T_{\lambda^b}
\neq 0\text{,}
\end{equation*}
where $\lambda^b$ is viewed as a $0$-vector field.  To prove the last equation, assume 
without loss of generality that $\lambda^b >\epsilon>0$ on a compact set of positive measure $S\subset \BBG$, so that we have
\begin{equation*}
T_{\lambda^b}(\charfcn_{S})= \int_S \lambda^b > \epsilon\mu(S)>0\text{.}
\end{equation*}

It follows from the previous paragraph that the vector field $\tilde{\lambda}$ in the expansion \eqref{weakexpansion} is uniquely determined up to null sets, so that $T_{\tilde{\lambda_1}}=T_{\tilde{\lambda_2}}$ if and only if $\tilde{\lambda_1} = \tilde{\lambda_2}$ almost everywhere.
\end{remark}

\subsection*{Smooth forms and smooth restrictions}
We have already defined the restriction of a current or precurrent by a metric form, or extended form.  We now discuss the case where a form is smooth.

\begin{definition}
\label{smoothformdef}
The elements of the subspaces $\smoothforms{k}(\BBG)=
\smooth(\BBG)^{k+1} \subset \forms{k}(\BBG)$  and $\smoothextforms{k}(\BBG) = \smooth(\BBG)\otimes \bigwedge^k\smooth(\BBG) \subset \extforms{k}(\BBG)$ are called \textbf{simple smooth forms} and \textbf{smooth forms}, respectively.
\end{definition}
If $\omega =f\,dg^1 \wedge \dotsb \wedge dg^k\in  \smoothforms{k}(\BBG)$, we denote by $\hat{\omega}$ the differential form $f\,\driem g^1\wedge \dotsb\wedge \driem g^k \in \diffforms{k}(\BBG)$.  
Because every differential $k$-form $\theta\in \diffforms{k}(\BBG)$ can be written 
\begin{equation*}
\theta = \sum_{a\in\Lambda(k,\dim(\BBG))} f_a\,\driem x_{a_1}\wedge\dotsb\wedge \driem x_{a_k}\text{,}
\end{equation*}
the map $\smoothextforms{k}(\BBG)\rightarrow \diffforms{k}(\BBG)$ given by $\omega\rightarrow \hat{\omega}$ is surjective.  Here we have implicitly invoked the fact that, as $\BBG$ is nilpotent, connected, and simply connected, the exponential map $\exp\colon \mathfrak g \rightarrow \BBG$ is a diffeomorphism, and $\mathfrak g$ in turn is diffeomorphic to $\BBR^{\dim(\BBG)}$. 
For an arbitrary manifold, of course, we could prove surjectivity by way of a partition of unity argument.

Note that if $\omega=f \,d g^1 \wedge \dotsb \wedge d g^k\in\smoothforms{k}(X)$, then equation \eqref{pansusmooth} implies that for every precurrent $T_{\tilde{\lambda}}$, we have
\begin{equation*}
T_{\tilde{\lambda}}\restrict{\omega}
= T_{\tilde{\lambda}\restrict{f \,\dpansu g^1 \wedge \dotsb \wedge \dpansu g^k}}
= T_{\tilde{\lambda}\restrict{f \,\driem g^1 \wedge \dotsb \wedge \driem g^k}}
= T_{\tilde{\lambda}\restrict{\hat{\omega}}}\text{.}
\end{equation*}
In particular, if $T$ is a precurrent and $\hat{\omega}_1=\hat{\omega}_2$, then $T\restrict{\omega_1}=T\restrict{\omega_2}$.
With this in mind we define the restriction of a precurrent by a smooth differential form.
\begin{definition}
\label{smoothrestrictions}
Let $T$ be a $k$-precurrent, and $\theta\in\diffforms{j}(\BBG)$.  We define the \textbf{restriction of $T$ by $\theta$} to be the $(k-j)$ precurrent 
\begin{equation*}
T\restrict{\theta} = T\restrict{\omega}\text{,}
\end{equation*}
where $\omega\in\smoothextforms{j}(\BBG)$ is any form such that $\hat{\omega} = \theta$.
We also define, for $\omega_1\in\extformsc{k_1}(\BBG)$, $\omega_2\in\extforms{k_2}(\BBG)$ and $\theta \in \diffforms{k_3}(\BBG)$, with $k_1+k_2+k_3 = k$, 
\begin{equation*}
T(\omega_1\wedge\theta\wedge\omega_2) = (-1)^{k_1 k_3} T\restrict{\theta}(\omega_1\wedge\omega_2)\text{.}
\end{equation*}
\end{definition}
Finally, we note that the extension to smooth forms applies to currents as well as precurrents.  To see this, suppose $\omega_1,\omega_2\in\smoothextforms{k}(\BBG)$, with $\hat{\omega_1}=\hat{\omega_2}=\theta\in\diffforms{k}(\BBG)$.  Then for any $T\in \currents{k}(\BBG)$, with $l\leq k$, with absolutely continuous mass $||T||$, $T$ is a precurrent, and so we have
\begin{equation}
\label{restrictsmooth}
T\restrict{\omega_1}=T\restrict{\omega_2}\text{.}
\end{equation}
By Proposition \ref{approximationlemma}, equation \eqref{restrictsmooth} extends to all currents in $\currents{k}(\BBG)$, making the restriction to  $\theta$ well-defined.




\subsection*{Normal currents}
If $T$ is a $k$-precurrent, we can define the boundary $\partial T$ as in Definition \ref{boundarydef}. It is not necessarily the case that $\partial T$ is also a precurrent --- the proof of Proposition \ref{currentsk} will provide a counterexample to this as well.  However, boundary continuity is closely related with the question of which precurrents are currents, as the following proposition indicates.

\begin{proposition}
\label{weaknormal}
Let $T$ be a  $k$-precurrent such that $\partial T$ is a $(k-1)$-precurrent.  Then $T\in\normalloc{k}(\BBG)$.
\end{proposition}

\begin{proof}

Multi-linearity and locality follow respectively from the linearity and locality of the Pansu-differentiation operator, and we have already shown precurrents have locally finite mass in Lemma \ref{prefinitemass} above.  All that remains, then, is to check that $T$ is continuous. 

We wish to show that for every sequence of forms $f_i \, dg^1_i \wedge \dotsb \wedge dg^k_i$ converging to $f \, dg^1 \wedge \dotsb \wedge dg^k\in\formsc{k}(\BBG)$, we have
\begin{equation}
\label{continuityeq}
\lim_{i\rightarrow \infty} T(f_i \, dg^1_i \wedge \dotsb \wedge dg^k_i) = T(f \, dg^1 \wedge \dotsb \wedge dg^k)\text{.}
\end{equation}

We first decompose the limit in \eqref{continuityeq}.
\begin{align}
& \lim_{i\rightarrow \infty} T(f_i \, dg^1_i \wedge \dotsb \wedge dg^k_i) \notag\\
=& \lim_{i\rightarrow \infty} T((f_i - f)\, dg^1_i \wedge \dotsb \wedge dg^k_i) +\lim_{i\rightarrow \infty}T(f \, d(g^1_i-g^1) \wedge dg^2_i \wedge \dotsb \wedge dg^k_i) \label{limitcomputation}\\
 &  +  \lim_{i \rightarrow \infty} T(f\, dg^1 \wedge dg^2_i \dotsb\wedge dg^k_i)\text{.}\notag
\end{align} 

By the locally finite mass condition mentioned above for $T$, the first term on the right-hand side of equation \eqref{limitcomputation} is $0$.  By way of the Leibniz rule, we next compute
\begin{align*}
& \lim_{i\rightarrow \infty}T(f \, d(g^1_i-g^1) \wedge dg^2_i \wedge \dotsb \wedge dg^k_i)\\
=& \lim_{i\rightarrow \infty}\partial T(f(g^1_i-g^1)\, dg^2_i \wedge \dotsb \wedge dg^k_i) - \lim_{i\rightarrow \infty}T((g^1_i-g^1)\,df \wedge dg^2_i \wedge \dotsb \wedge dg^k_i)\\
=& 0  \text{.}
\end{align*}
Here both terms in the second line vanish on account of the locally finite mass condition, since $T$ and $\partial T$ are both precurrents.

Equation \eqref{limitcomputation} then reduces to
\begin{equation}
\label{limcomp1}
\lim_{i\rightarrow \infty} T(f_i \, dg^1_i \wedge \dotsb \wedge dg^k_i) = \lim_{i\rightarrow \infty} T(f\, dg^1 \wedge dg^2_i \wedge \dotsb\wedge dg^k_i)\text{.}
\end{equation}
Moreover, by the alternating property, from equation \eqref{limcomp1} we deduce that for $1\leq j \leq k$,  
\begin{equation}
\label{limcomp1eq}
\lim_{i\rightarrow \infty} T(f_i \, dg^1_i \wedge \dotsb \wedge dg^k_i) = \lim_{i\rightarrow \infty} T(f\, dg_i^1 \wedge \dotsb dg^{j-1}_i\wedge dg^j \wedge dg^{j+1}_i\wedge\dotsb\wedge dg^k_i)\text{.}
\end{equation}
Applying equation \eqref{limcomp1eq} successively for $j=1,\dotsc,k$ yields equation \eqref{continuityeq}.
\end{proof}

\subsection*{1-currents}

Though we will see shortly that a $k$-precurrent need not be a current when $k\geq 2$, $1$-precurrents are always currents. 
\begin{lemma}
\label{weak1}
Every $1$-precurrent in $\BBG$ is a current.
\end{lemma}

Lemma \ref{weak1} follows from a simple observation, which will itself be of use momentarily, in the proof of Proposition \ref{currentsk}.  
\begin{lemma}
\label{noboundary}
Let $\BBG$ be a Carnot group.  Then an invariant $1$-precurrent has vanishing boundary.
\end{lemma}
\begin{proof}
We must show that for any $f\in \lipc(\BBG)$, and any $u\in \mathfrak h$, 
\begin{equation}
\label{1currentvanishing}
T_u(df)= \int_{\BBG} X^u(f) = 0\text{.}
\end{equation}
Without loss of generality, we will assume $||u||=1$. Since $\BBG$ is unimodular, we recall from the theory of topological groups (see, e.g, \cite[Theorem 6.18]{Knapp}) that for any unimodular subgroup $S\subset \BBG$, with Haar measure $\mu_S$, there is a left-invariant measure $\mu_{\BBG/S}$ on the quotient $\BBG/S$ such that for any $g\in \contc(\BBG)$, we have
\begin{equation}
\label{haarfubini}
\int_{\BBG} g \,d\mu = \int_{\BBG/S} \left(\int_S g(ps)\,d\mu_S(s)\right)\,d\mu_{\BBG/S}(pS)\text{.}
\end{equation}

We apply equation \eqref{haarfubini} with $g=X^u(f)$, $S=\exp(\spa \{u\})$, and $d\mu_S= ds$, where $ds$ is the arc-length measure.  Notice that since $||u||=1$, the map $\gamma(t)= p\cdot\exp(tu)$ is an isometry, as well as an integral curve of the vector field $X^u$. We thus have
\begin{equation*}
\int_{S} X^u(f)\,ds
= \int_{\BBR} X^u_{(p\cdot\exp(tu))}(f)\,dt
= \int_{\BBR} \frac{d}{dt}(f\circ\gamma)\,dt
=0\text{,}
\end{equation*}
since $f$ has compact support.  Now equation \eqref{1currentvanishing} follows from equation \eqref{haarfubini}.
\end{proof}
\begin{proof}[Proof of Lemma \ref{weak1}]
It is enough to show that invariant $1$-precurrents are actually currents.  Indeed, once we have proved this, we see that each of the precurrents $T_{u_i}$ are $1$-currents.  But restricting a $1$-current by a function or form, as in Definition \ref{restriction}, gives us another current.  Thus every precurrent of the form $T=\sum_{i=1}^n T_{u_i}\restrict{\lambda^i}$ is a current.  By equation \eqref{weakexpansion}, every $1$-precurrent has this form.

The proof now follows from Lemma \ref{noboundary} and Proposition \ref{weaknormal}. 
\end{proof}

Though we will see that precurrents need not be currents, the following corollary to Lemma \ref{weak1} shows that precurrents are separately continuous in each variable.  
\begin{corollary}
\label{separatecontinuity}
Let $T$ be a $k$-precurrent, let $\omega_i = f\, dg^1 \wedge\dotsb\wedge dg^{j-1}\wedge dg^j_i\wedge dg^{j+1}\wedge \dotsb\wedge dg^k\in\formsc{k}(\BBG)$ be a sequence of forms, let $\omega=f\,dg^1\wedge\dotsb\wedge dg^k\in\formsc{k}(\BBG)$, and suppose $g^j_i$ converges to $g^j$ in $\liploc(\BBG)$.  Then
\begin{equation}
\label{sepcont}
\lim_{i\rightarrow\infty}T(\omega_i)= T(\omega)\text{.}
\end{equation}
Equation \eqref{sepcont} holds as well for the case $j=0$, if $f_i$ converges to $f$ in the topology of $\lipc(X)$, where $f_i=g^0_i$, and $f=g^0$.
If $T$ is a $(k+1)$-precurrent, then equation  \eqref{sepcont} holds when $T$ is replaced by $\partial T$
\end{corollary}

\begin{proof}
The restriction of $T$ to a metric$(j-1)$-form is a $1$-precurrent, and hence a current by Lemma \ref{weak1}.  Thus
\begin{align*}
&\lim_{i\rightarrow\infty} T(\omega_i) 
= \lim_{i\rightarrow \infty} (-1)^{k-j} T\restrict{dg^1 \wedge\dotsb\wedge dg^{j-1}\wedge dg^{j+1}\wedge \dotsb\wedge dg^k}(f\,dg^j_i)\\
&= (-1)^{k-j} T\restrict{dg^1 \wedge\dotsb\wedge dg^{j-1}\wedge dg^{j+1}\wedge \dotsb\wedge dg^k}(f\,dg^j)
= T(\omega)\text{.}
\end{align*}
The continuity in the variable $f$ follows from the same argument.  
The argument for $\partial T$ is identical.
\end{proof}

As a consequence of Corollary \ref{separatecontinuity} and Lemma \ref{smoothdense}, two precurrents are equal if they are equal when evaluated on smooth forms, and similarly for boundaries of precurrents.  
\begin{corollary}
\label{smootheval}
Suppose $T_1$ and $T_2$ are $k$-precurrents, and that 
\begin{equation}
\label{evaleq}
T_1(\omega)=T_2(\omega)
\end{equation}
for any smooth form $\omega\in\smoothformsc{k}(\BBG)$.
Then $T_1=T_2$. Similarly, if $\partial T_1(\omega')= \partial T_2(\omega')$ for every $\omega'\in\smoothformsc{k-1}(\BBG)$, then $\partial T_1=\partial T_2$.
\end{corollary}
\begin{proof}
Suppose there is a number $j$, $0\leq j\leq k+1$, such that equation \eqref{evaleq} holds whenever $\omega= f\, dg^1\wedge\dotsb\wedge dg^k$, where  $g^m\in\smooth(\BBG)$ whenever $j\leq m \leq k$  (letting $g^0=f$).  We claim then that the same is true for $j+1$.  Indeed, by Lemma \ref{smoothdense}, there is a sequence of smooth functions $g^j_i$ converging to $g^j$ in $\liploc(\BBG)$.  Then by Corollary \ref{separatecontinuity}, we have
\begin{equation*}
\partial T(f\wedge dg^1\wedge\dotsb\wedge dg^k) =
\lim_{i\rightarrow\infty} T(df\wedge dg^1\wedge\dotsb\wedge dg^j_i \dotsb \wedge dg^k)= 0\text{.}
\end{equation*}
The result now follows by induction on $j$, as the case $j=0$ is true by hypothesis, and the case $j=k+1$ is a restatement of the corollary.  The last statement is proved with the same argument.
\end{proof}

%% file: Invariantcurrents.tex
\section{Invariant currents.}
\label{invariantcurrents}

In this section we prove Proposition \ref{currentsk}, which characterizes translation invariant currents.  
From the definition, a precurrent $T=T_{\tilde{\lambda}}$ is invariant if and only if $\tilde{\lambda}\circ \tau_p = \tilde{\lambda}$ almost everywhere, which in turn occurs if and only if $\tilde{\lambda}$ is 
constant almost everywhere.  

To formulate Proposition \ref{currentsk}, we will need the notion of a  ``vertical form''.  We will call a differential $1$-form $\theta\in\diffforms{1}(\BBG)$ \textit{vertical} if $T\restrict{\theta}=0$ for every $k$-precurrent $T$.  Equivalently, $\theta$ is vertical if and only if $\theta$ annihilates every horizontal vector field, i.e., $\langle\theta,X^u\rangle\equiv 0$ for every $u \in H$.

\begin{example}
\label{heisvert}
In the $n^{\text{th}}$ Heisenberg group $\heis{n}$, the  basis $X_1$,\ldots,$X_n$,$Y_1$,\ldots, $Y_n$, $Z$ has a dual basis  consisting of forms $dx_1$,\ldots, $dx_n$, $dy_1$,\ldots,$dy_n$, $\theta$.  The form $\theta$ is a vertical form, as it vanishes when paired with every horizontal vector field.  $\theta$ is sometimes called the \textit{contact form}, as $(\heis{n},\theta)$ is a \textit{contact manifold}, meaning that $\theta\wedge (d\theta)^{\wedge n}$ is a volume form on $\heis{n}$.

It can be shown \cite[Section 2]{Rumin}
that $\theta$ and $d\theta$ generate $\diffforms{k}(\heis{n})$ for $k>n$; that is, every $\omega \in \diffforms{k}(\heis{n})$ has the form $\omega = \alpha\wedge \theta +\beta\wedge d\theta$.
\end{example}

The following lemma describes the push-forwards of an invariant precurrent along the dilation maps $\Delta_r$.

\begin{lemma}
\label{dilate}
Let $T=T_{\tilde{\lambda}}$ be an invariant $k$-precurrent.  Then $\Delta_{r\#}T = r^{k-Q} T$.
\end{lemma}
\begin{proof}
We compute, via equations \eqref{pushforwardcarnot} and \eqref{pansudilation},
\begin{align*}
&\Delta_{r\#}T(f\,dg^1\wedge\dotsb\wedge dg^k)
= \int_\BBG \langle  f(\Delta_r(p)) \, \dpansu (g^1\circ \Delta_r)_p\wedge \dotsb\wedge \dpansu (g^k\circ\Delta_r)_p,\tilde{\lambda}\rangle \, d\mu(p)\\
&=r^k\int_\BBG \langle f(\Delta_r(p)) \, \dpansu g^1_{\Delta_r(p)} \wedge \dotsb\wedge  \dpansu g^k_{\Delta_r(p)},\tilde{\lambda}\rangle \, d\mu(p)\\
&= r^k\int_\BBG \langle  f(p) \, \dpansu g^1_p \wedge \dotsb\wedge  \dpansu g^k_p,\tilde{\lambda}\rangle \, d\Delta_{r\#}\mu(p)\\
&= r^{k-Q}\int_\BBG \langle  f(p) \, \dpansu g^1_p \wedge \dotsb\wedge  \dpansu g^k_p,\tilde{\lambda}\rangle \, d\mu(p)
= r^{k-Q}T(f\,dg^1\wedge\dotsb\wedge dg^k)\text{.}
\end{align*}
\end{proof}

\begin{proposition}
\label{currentsk}
Let $T$ be an invariant $k$-precurrent in a Carnot group $\BBG$.  The following statements are equivalent:
\begin{enumerate}
\item $T$ is a current.
\item $\partial T = 0$.
\item $T\restrict{d\theta}=0$ for every vertical $1$-form $\theta\in \diffforms{k}(\BBG)$.
\item $T\restrict{d\theta}=0$ for every invariant vertical $1$-form $\theta\in \diffforms{k}(\BBG)$.
\end{enumerate}
\end{proposition}

\begin{proof}
We prove $1\Leftrightarrow 2$, and $ 2\Rightarrow 3 \Rightarrow 4\Rightarrow 2$.

$1\Rightarrow 2$:  Suppose $T$ is an invariant $k$-current.  We wish to show that $\partial T=0$, that is, for $g^1\,dg^2\wedge\dotsb\wedge dg^k\in\formsc{k-1}(X)$,  $T(dg^1\wedge\dotsb\wedge dg^k)=0$.  

We may assume without loss of generality that $T(dg^1\wedge\dotsb\wedge dg^k)\geq 0$.  We also assume 
that each function $g_i$ has compact support, and hence all of the functions are supported in some ball $B_R(\gunit)$ centered at the identity.  

For every $\epsilon>0$, we define the rescaled functions $g^i_\epsilon$ by
\begin{equation*}
g^i_\epsilon(p)=\epsilon g^i\circ\Delta_{\frac{1}{\epsilon}} \text{,}
\end{equation*}
and note that $g^i_\epsilon$ is supported on $B_{R/\epsilon}$ and has the same Lipschitz constant as $g^i$. 
Also, we let $N\subset B_{R/2}(0)$ be a maximal $4 \epsilon R$-separated subset of the ball $B_{R/2}(0)$.  By the $Q$-regularity of $\BBG$, $\#N\geq C\epsilon^{-Q}$.  We define the functions $\hat{g}^i_\epsilon$ by
\begin{equation*}
\hat{g}^i_\epsilon = \sum_{p\in N} g^i_\epsilon\circ\tau_p\text{.}
\end{equation*}
Again, we note that $\hat{g}^i_\epsilon$ has the same Lipschitz constant as $g^i$.  Moreover, for $p, q\in N$, $p\neq q$, we have $\supp(g_\epsilon^i\circ\tau_p)\cap\supp(g_\epsilon^j\circ\tau_p)=\emptyset$, and so by 
the invariance of $T$ under left translations,
\begin{equation*}
T(d g^1_\epsilon\circ\tau_{p_1} \wedge\dotsb\wedge  d g^k_\epsilon\circ\tau_{p_k})=
\begin{cases} 
T(d g^1_\epsilon \wedge\dotsb\wedge  d g^k_\epsilon)& \text{if $p_1=\dotsb=p_k$,}
\\
0 &\text{otherwise.}
\end{cases}
\end{equation*}
But now, with the help of Lemma \ref{dilate}, we compute 
\begin{align*}
&T(d\hat{g}^1_\epsilon\wedge\dotsb\wedge d\hat{g}^k_\epsilon) = \sum_{p\in N} T(d g^1_\epsilon \wedge\dotsb\wedge  d g^k_\epsilon)
= \# N \cdot T(d g^1_\epsilon \wedge\dotsb\wedge  d g^k_\epsilon)\\
&=\# N \cdot \epsilon^k T(d (g^1\circ\Delta_{\frac{1}{\epsilon}}) \wedge\dotsb\wedge  d( g^k\circ\Delta_{\frac{1}{\epsilon}}))
= \# N \cdot \epsilon^k \Delta_{\frac{1}{\epsilon}\#}T(d g^1 \wedge\dotsb\wedge  d g^k)\\
&= \# N \cdot \epsilon^Q T(d g^1 \wedge\dotsb\wedge  d g^k)
\geq C T(d g^1 \wedge\dotsb\wedge  d g^k)\text{.}
\end{align*}

As $\epsilon$ approaches $0$, the functions $\hat{g}^i_\epsilon$ converge to $0$ uniformly and with bounded Lipschitz constant, so $C T(d g^1 \wedge\dotsb\wedge  d g^k)$ must approach $0$ by continuity of $T$.  Since the functions $g^i$ are independent of $\epsilon$, $T(d g^1 \wedge\dotsb\wedge  d g^k)=0$.

$2\Rightarrow 1$:  This follows immediately from Proposition \ref{weaknormal}.

$2\Rightarrow 3$:  If $\partial T=0$, then for any $f\,dg^1\wedge\dotsb\wedge dg^{k-2} \in \smoothformsc{k-2}(\BBG)$ and any vertical $\theta\in \diffforms{1}(\BBG)$, we have
\begin{align*}
&\phantom{=}\;T\restrict{d\theta}(f\,dg^1\wedge\dotsb\wedge dg^{k-2}) = T(f\, d\theta \wedge dg^1\wedge\dotsb\wedge dg^{k-2})\\
&= T(d(f\theta) \wedge dg^1\wedge\dotsb\wedge dg^{k-2})-T(df\wedge \theta \wedge dg^1\wedge\dotsb\wedge dg^{k-2})\\
&= \partial T(f\theta\wedge dg^1\wedge\dotsb\wedge dg^{k-2}) + T\restrict{\theta}(df\wedge dg^1\wedge\dotsb\wedge dg^{k-2})
= 0\text{.}
\end{align*}
Here the second term vanishes because $\theta$ is vertical.

$3\Rightarrow 4$ is clear.

$4\Rightarrow 2$:  Let $T=T_{\tilde{\lambda}}=\sum_{a\in\Lambda(k,n)} T_{\lambda^a \tilde{u}_a}$, where here, since $T$ is invariant, each $\lambda^a$ is constant.  
For $\omega=f\,dg^1\wedge \dotsb\wedge dg^k\in \smoothformsc{k}(\BBG)$, and each $a\in\Lambda(k,n)$, we compute  
\begin{align}
&\partial T_{\tilde{u}_a}(\omega) = T_{\tilde{u}_a}(d\omega)
= \int_{\BBG} \langle  d\omega,\tilde{u}_a \rangle \notag\\
=& \int_{\BBG} \sum_{i=1}^k (-1)^{i-1} X^{u_{a_i}} \left(\langle \omega,u_{a_1}\wedge\dotsb\wedge \widehat{u_{a_i}}\wedge\dotsb\wedge u_{a_k}\rangle\right)\,d\mu\notag\\
  &+ \int_{\BBG} \sum_{i<j} (-1)^{i+j}\langle\omega,[u_{a_i},u_{a_j}]\wedge u_{a_1} \dotsb \wedge \widehat{u_{a_i}}\wedge\dotsb\wedge \widehat{u_{a_j}}\wedge \dotsb\wedge u_{a_k}\rangle d\mu \label{boundarycomputation}\\
=&  \sum_{i=1}^k (-1)^{i-1}  \partial T_{u_{a_i}}\left(\langle \omega,u_{a_1}\wedge\dotsb\wedge \widehat{u_{a_i}}\wedge\dotsb\wedge u_{a_k}\rangle\right)\notag\\
  &+ \sum_{i<j} (-1)^{i+j} \int_\BBG \langle\omega,[u_{a_i},u_{a_j}]\wedge u_{a_1} \dotsb \wedge \widehat{u_{a_i}}\wedge\dotsb\wedge \widehat{u_{a_j}}\wedge \dotsb\wedge u_{a_k}\rangle d\mu \text{.}\notag
\end{align}

See \cite[Proposition 3.2]{LangS}, e.g., for the expansion in the second and third lines of \eqref{boundarycomputation}; here, as in \cite{LangS}, the symbol ``$\widehat{\phantom{X}}$'' above a vector means that vector should be omitted.)

By Lemma \ref{noboundary}, $\partial T_{u_{a_i}}=0$ for all $i$, so the first sum in the last line vanishes.  Since $[u_{a_i},u_{a_j}]\in V_2$ for all $i$ and $j$, expanding the second sum in the last line shows that
the boundary $\partial T_{\tilde{u}_a}$ satisfies
\begin{equation}
\label{boundarycomponentk}
\partial T_{\tilde{u}_a} (\omega) =  \sum_{b \in \Lambda_{n,k-2}} \int_\BBG  \langle\omega,v_b \wedge u_{b_1} \wedge \dotsb \wedge u_{b_{k-2}}\rangle \text{,} 
\end{equation}
where each $v_b =v_b(a) \in V_2$.  Again, this holds for every $\omega\in\smoothformsc{k}(\BBG)$.  Of course, the vectors $v_b$ do not depend on the choice of $\omega$. 

Since $T=\sum_{b\in\Lambda (k,n)} \lambda^b T_{\tilde{u}_b}$,  and the boundary operator is linear, 
it follows that $\partial T(\omega)$ can be written in the form of equation \eqref{boundarycomponentk}.  
\begin{equation}
\label{boundaryk}
\partial T (\omega) =  \sum_{b \in \Lambda_{n,k-2}} \int_\BBG  \langle \omega,v_b \wedge u_{b_1} \wedge \dotsb \wedge u_{b_{k-2}}\rangle \text{.} 
\end{equation}
Thus $\partial T$, when applied to a smooth form, is given by integration against an invariant $(k-1)$-vector field in $V_2\wedge \left(\bigwedge^{k-2}H\right)$.

Now suppose $T\restrict{d\theta}=0$ for every smooth invariant $1$-form $\theta\in\diffforms{1}(\BBG)$, and recall that $\{\dpansu \pi^1,\dotsc,\dpansu \pi^n\}$ is the dual basis to $\{u_1,\dotsc,u_n\}$.
Then for every such $\theta$,
every $a\in \Lambda_{n,k-2}$, and every $f\in\smoothc(\BBG)$, we have
\begin{align*}
&0 = T\restrict{d\theta}\restrict{\dpansu \pi^{a_1}\wedge \dotsb \wedge \dpansu \pi^{a_{k-2}}}(f)
= (\partial T\restrict{\theta}+\partial(T\restrict{\theta}))\restrict{\dpansu \pi^{a_1}\wedge \dotsb \wedge \dpansu \pi^{a_{k-2}}}(f)\\
&= \partial T(f \theta\wedge \dpansu \pi^{a_1}\wedge \dotsb \wedge  \dpansu \pi^{a_{k-2}})\\
&= \sum_{b \in \Lambda_{n,k-2}}  \int_\BBG f\langle \theta \wedge \driem \pi^{a_1}\wedge \dotsb \wedge \driem \pi^{a_{k-2}},v_b \wedge u_{b_1} \wedge \dotsb \wedge u_{b_{k-2}}\rangle \,d\mu\\
&= \sum_{b \in \Lambda_{n,k-2}}  \int_\BBG f \delta_a^b \langle \theta,v_b\rangle
= \int_\BBG f  \langle \theta,v_a\rangle\text{.}
\end{align*}
Here $\delta_a^b=1$ if and only if $a=b$, and $0$ otherwise.  Since this holds for all $f$, and in particular, any nonzero, nowhere negative $f\in\smoothc(\BBG)$, we have $\langle \theta,v_a\rangle=0$.
If $v_a\neq 0$, then $v_a \notin H$, so there is an invariant vertical $1$-form $\theta$ such that $ \theta,\langle v_a \rangle \neq 0$, a contradiction. Thus $v_a=0$.
This holds for all $a\in \Lambda(k,n)$, so by equation \eqref{boundaryk}, we have $\partial T(\omega)=0$ for all $\omega\in \smoothforms{k-1}(\BBG)$.
Therefore, by Corollary \ref{smootheval}, $\partial T = 0$.
\end{proof}

%% file: Generalcurrentsincarnotgroups.tex
\section{General currents in Carnot groups.}
\label{generalcarnotcurrents}

In this section we prove Theorem \ref{characterization}.  We need to relate arbitrary precurrents to invariant ones, and so we introduce a kind of tangent approximation.  Let $T=T_{\tilde{\lambda}}$ be a $k$-precurrent.  At a given point $p\in \BBG$, we define the current $T^p$ by the equation
\begin{equation*}
T^p = T_{\tilde{\lambda}_p}\text{.}
\end{equation*}
Note that $T^p$ is well-defined up to sets of measure $0$.


\begin{lemma}
\label{pointwisecurrent}
A $k$-precurrent $T$ is a current if and only $T^p$ is a current for almost every $p\in \BBG$. 
\end{lemma}

\begin{proof}
Let $T=T_{\tilde{\lambda}}$, and suppose first that for almost every $p\in \BBG$, $T^p$ is a current. Now suppose that $p$ is a Lebesgue point of each function $\lambda^a$ for $a\in\Lambda(k,n)$. Note that since each $\lambda^a$ is locally integrable, almost every $p\in\BBG$ satisfies this condition.
For every $\epsilon>0$, there is a number $R=R(\epsilon,p)>0$ such that for $0\leq r\leq R$, we have
\begin{equation*}
\sum_{a\in\Lambda(k,n)}\int_{B_r(p)}|\lambda^a-\lambda^a_p|\,d\mu \leq \epsilon \mu(B_r(p))\text{.}
\end{equation*}
Thus by inequality \eqref{massestimate},
\begin{equation}
\label{invapprox}
||T-T^p||(B_r(p)) \leq \frac{n!}{k!} \sum_{a\in\Lambda(k,n)} \int_{B_r(p)}  |\lambda^a-\lambda^a_p|\,d\mu
\leq \frac{n!}{k!} \epsilon \mu(B_r(p))\text{.}
\end{equation}
Since equation \eqref{invapprox} holds for almost every $p\in\BBG$, and every $r<R(\epsilon,p)$, by the Vitali Covering Theorem, there is a countable pairwise disjoint collection of balls $B_i = B_{r_i}(p_i)$ such that $\mu(\BBG\backslash \bigcup_{i=1}^\infty B_i)=0$, and such that $r_i\leq \min(\epsilon,R(\epsilon,p_i))$.

Let $T_\epsilon = \sum_{i=1}^\infty T^{p_i} \restrict{B_i}$.  We claim this sum converges locally in mass.   Indeed, given a relatively compact subset $U\subset \BBG$, let $U_\epsilon=\{q\in\BBG:\dist(U,q)<\epsilon\}$.  Then
\begin{align*}
&\sum_{i=1}^\infty ||T^{p_i} \restrict{B_i}||(U)
\leq \sum_{p_i\in U_\epsilon} ||T^{p_i} \restrict{B_i}||(U)
\leq \sum_{p_i\in U_\epsilon} ||T^{p_i}||(B_i)\\
&\leq \sum_{p_i\in U_\epsilon} \left(||T||(B_i)+\epsilon\frac{n!}{k!}\mu(B_i)\right)
\leq  ||T||(U_\epsilon)+\epsilon \mu(U_\epsilon)\text{,}
\end{align*}
and so the sum converges.

Moreover, we have
\begin{equation*}
||T_\epsilon-T||(U)
\leq \sum_{p_i\in U_\epsilon} ||T-T^{p_i}||(B_i)
\leq \sum_{p_i\in U_\epsilon} \epsilon\frac{n!}{k!}\mu(B_i)
\leq \epsilon \frac{n!}{k!}\mu(U_\epsilon)\text{.}
\end{equation*}
Thus $T_\epsilon\restrict{U}$ converges to $T\restrict{U}$ in mass as $\epsilon$ approaches $0$.  Since each $T_\epsilon\restrict{U}$ is a current, $T\restrict{U}$ is also a current, by the completeness of the space of currents in the mass norm.  Being a current is a local property (indeed, the continuity axiom is satisfied for $T$ if and only if it is satisfied for $T\restrict{U}$ for every relatively compact open set $U\subset \BBG$), and so $T$ is a current. 

Conversely, suppose that $T$ is a current, and let $p$ be a Lebesgue point for each $\lambda^a$ as above.  Let $\Delta^p_t$ denote the dilation centered at point $p$, that is $\Delta^p_t = \tau_p\circ\Delta_t\circ\tau_{p^{-1}}$. Let $R_0>0$, and let $R=R(\epsilon,p)$ be as above.  Finally, let $s=\frac{R_0}{R}$.

From Lemma \ref{dilate}, and the invariance of $T^p$ under translations, we have
$\Delta^p_{s\#}T = s^{k-Q} T$.  Moreover, since the dilation map $\Delta^p_s$ scales distances precisely by a factor of $s$, it follows from Definition \ref{massdef} that for any $k$-precurrent $T'$, $||\Delta^p_{s\#} T'||(B_(R_0)(p)) = s^k ||T'||(B_R(p))$ (see also \cite[Lemma 4.6 (2)]{Lang} for the same argument applied to currents). Thus by equation \eqref{invapprox},
\begin{align*}
&||s^{Q-k}\Delta^p_{s\#} T - T^p ||(B_{R_0}(p)) =  s^{Q-k}||\Delta^p_{s\#} (T - T^p) ||(B_{R_0}(p))\\
&= s^{Q-k}s^k ||T - T^p||(B_R(p))
\leq s^Q \epsilon \mu(B_R(p))
= \epsilon \mu(B_{R_0}(p)) \text{.}
\end{align*}
The currents $s^{Q-k}\Delta^p_{s\#} T$ therefore converge locally in mass to $T^p$, and so $T^p$ is a current.
\end{proof}

We are now ready to prove Theorem \ref{characterization}.


\begin{proof}[Proof of Theorem \ref{characterization}]
Let $T=T_{\tilde{\lambda}}$ be a $k$-precurrent.  Since by definition, all precurrents satisfy $T(\theta\wedge\alpha) = 0$ for every vertical $\theta\in \diffforms{1}(\BBG)$, equation \eqref{chareq} is equivalent to the condition that $T\restrict{\theta}=0$.

If $T$ is a current, then by Lemma \ref{pointwisecurrent}, $T^p$ is an invariant current for almost every $p\in\BBG$.  Let $T_\epsilon$ be as in the proof of Lemma \ref{pointwisecurrent}.  Then for any $\omega\in\formsc{k-2}(\BBG)$, and any vertical $\theta\in\diffforms{1}(\BBG)$, by Proposition \ref{currentsk} we have
\begin{equation*}
T_\epsilon\restrict{d \theta}(\omega)
=\sum_{i=1}^\infty T^{p_i} \restrict{B_i}\restrict{d\theta}(\omega)\\
=\sum_{i=1}^\infty T^{p_i}\restrict{d\theta} \restrict{B_i}(\omega)\\
= 0\text{.}
\end{equation*}
Since $T_\epsilon$ converges locally in mass, and hence weakly, to $T$, we have $T\restrict{d\theta}(\omega)=0$.

Conversely, suppose $T$ is a $k$-precurrent, with $T\restrict{d\theta}=0$ for every vertical form $\theta$.  Notice that if $\theta$ is vertical, then so is $\Delta^{p*}_t\theta$, since translations and dilations both respect the horizontal bundle.
Thus for all $t>0$, we have
\begin{equation*}
(\Delta^p_{t\#}T)\restrict{d\theta} = \Delta^p_{t\#}(T\restrict{\Delta^{p*}_t d\theta})=  \Delta^p_{t\#}(T\restrict{d(\Delta^{p*}_t \theta)}) =      0\text{.}
\end{equation*}
Since, at almost every $p\in\BBG$, the currents $s^{Q-k}\Delta^p_{s\#} T$ from the second half of Lemma \ref{pointwisecurrent} converge locally in mass to $T^p$, it follows that $T^p\restrict{d\theta}=0$.  Thus by Proposition \ref{currentsk}, $T^p$ is a current, and by the first half of Lemma \ref{pointwisecurrent}, so is $T$. 

For the last statement of the Theorem, suppose that $T\in \currentsloc{k}(\BBG)$, and $\theta\in\diffforms{1}(\BBG)$, is vertical. 
By Proposition \ref{approximationlemma}, $T$ can be approximated weakly by currents $T_\epsilon$ of absolutely continuous mass.  By Theorem \ref{vectormeasure}, each current $T_\epsilon$ is also a precurrent, so that $T_\epsilon\restrict{\theta}=0$, and from the first part of the theorem,  $T_\epsilon\restrict{d\theta}=0$ as well. 
Thus we have
\begin{equation*}
T\restrict{d\theta}(\omega)
= T(d\theta\wedge\omega)
= \lim_{\epsilon\rightarrow 0} T_\epsilon(d\theta\wedge\omega)
= \lim_{\epsilon\rightarrow 0} T_\epsilon\restrict{d\theta}(\omega)
= 0\text{,}
\end{equation*}
and the same computation shows that $T\restrict{\theta}(\omega)=0$.
\end{proof}

%% file: Normalcurrentsincarnotgroups.tex
\section{Normal currents in Carnot groups}
\label{normalsurjection}
In this section we prove Theorem \ref{normalisomorphism}.  We first establish a preliminary result to allow us to restrict our attention to smooth forms.
\begin{proposition}
\label{smoothextension}
Let $T\colon \smoothforms{k}(\BBG)\rightarrow \BBR$ satisfy the linearity and locality axioms in Definition \ref{currentdef}, with respect to the Carnot-Carath\'eodory metric $\cdist$.  Assume further that $T$ and $\partial T$ satisfy the finite mass condition in Definition \ref{massdef}, where the Lipschitz functions in the definition are required to be smooth.  Then $T$ has a unique extension $\hat{T}\in \normalloc{k}(\BBG)$.
\end{proposition}
\begin{proof}
Let $\omega=f\,dg^1\wedge\dotsb\wedge dg^k\in \formsc{k}(\BBG)$.  Let $\omega_\epsilon= f_\epsilon\,dg^1_\epsilon\wedge\dotsb\wedge dg^k_\epsilon\in \smoothformsc{k}(\BBG)$, where $f_\epsilon$ and $g^j_\epsilon$ are constructed as in the proof of Lemma \ref{smoothdense}.  We claim $T(\omega_\epsilon)$ converges as $\epsilon$ converges to $0$.  Indeed, for every $\epsilon>0$ and $\epsilon_1, \epsilon_2\in (0,\epsilon)$, we have
\begin{align}
&|T(\omega_{\epsilon_1})-T(\omega_{\epsilon_2})|
\leq  |T((f_{\epsilon_1}-f_{\epsilon_2})\,dg^1_{\epsilon_1}\wedge\dotsb\wedge dg^k_{\epsilon_1})| \notag\\
& + \sum_{j=1}^{k}|T(f_{\epsilon_2}\,dg^1_{\epsilon_2}\wedge\dotsb\wedge d(g^j_{\epsilon_1}-g^j_{\epsilon_2})\wedge\dotsb\wedge dg^k_{\epsilon_1})|\notag\\
&= |T((f_{\epsilon_1}-f_{\epsilon_2})\,dg^1_{\epsilon_1}\wedge\dotsb\wedge dg^k_{\epsilon_1})|\label{cauchyexp}\\
&+ \sum_{j=1}^{k} |\partial T(f_{\epsilon_2} (g^j_{\epsilon_1}-g^j_{\epsilon_2}) \,dg^1_{\epsilon_2}\wedge\dotsb\wedge dg^{j-1}_{\epsilon_2}\wedge dg^{j+1}_{\epsilon_1}\wedge\dotsb\wedge dg^k_{\epsilon_1})|\notag\\
&+\sum_{j=1}^k|T((g^j_{\epsilon_1}-g^j_{\epsilon_2})\,dg^1_{\epsilon_2}\wedge\dotsb\wedge df_{\epsilon_2}\wedge\dotsb\wedge dg^k_{\epsilon_1})|
\text{.}\notag
\end{align}
By the finite mass assumption on $T$ and $\partial T$, and the fact that the sequences $\{f_\epsilon\}$ and $\{g^j_\epsilon\}$ converge to $f$ and $g^j$ (respectively), with the same respective Lipschitz constants, each term on the right hand side of equation \eqref{cauchyexp} converges to zero with $\epsilon$, independently of the choice of $\epsilon_1$ and $\epsilon_2$.  Thus $T(\omega_\epsilon)$ converges, and so we let $\hat{T}(\omega)=\lim_{\epsilon\rightarrow 0} T(\omega_\epsilon)$.  

By the linearity of the convolution operators in the proof of Lemma \ref{smoothdense}, and the linearity of $T$, $\hat{T}$ satisfies the linearity axiom.

If $g^j$ is constant on a neighborhood of $\supp(f)$, say $N_{\epsilon_0}(\supp(f))$, then for every $\epsilon<\epsilon_0$, $g^j_\epsilon$ is constant on $\supp(f)$.  By the locality property of $T$, $T(\omega_\epsilon)=0$ for such $\epsilon$, so that $T(\omega)=0$.  From Remark \ref{localityremark}, we may conclude that $\hat{T}$ satisfies the locality axiom as well. 

We next claim that $\hat{T}$ has locally finite mass.  Indeed, suppose $U\subset \BBG$ is open, and $\omega\in\extformsc{k}(U)$ with $||\omega||\leq 1$.  From the definition, we may write
\begin{equation*}
\omega = \sum_{s\in S} f_s\, dg^1_s \wedge \dotsb \wedge dg^k_s
\end{equation*}
for some functions $g^i_s\in\liploc(X)$ such that $L(g^i_s|_{\supp(f)})\leq 1$.  Moreover, by Remark \ref{localityremark}, we may assume without loss of generality that $L(g^i_s)=1$, since we may replace $g^i_s$ with $\tilde{g^i_s}$, where $\tilde{g^i_s}$ is a McShane extension of $g^i_s|_{\supp(f)}$ to $\BBG$ (see, e.g., \cite[Theorem 6.2]{Heinonen}. With this assumption, we have $||\omega_\epsilon||\leq 1$ for each $\epsilon>0$.  Since $U$ is open and $\supp(f)$ is compact, $f_\epsilon$ is supported in $U$ for sufficiently small $\epsilon$, so that $|T(\omega_\epsilon)|\leq ||T||(U)$ and $|\partial T(\omega_\epsilon)|\leq ||\partial T||(U)$.  By the construction of $T$, and the locally finite mass assumptions on $T$ and $\partial T$, $||\hat T(U)||\leq ||T(U)||< \infty$ and $||\partial\hat{T}||(U)\leq ||\partial T(U)||<\infty$.

The continuity axiom follows immediately from the local finiteness of $||T||$ and $||\partial T||$, via the decomposition \eqref{cauchyexp}, with $\omega_{\epsilon_1}$ and $\omega_{\epsilon_2}$ replaced by $\omega_i$ and $\omega$, respectively.

Finally, uniqueness of $T$ is a consequence of Corollary \ref{smootheval}.
\end{proof}

\begin{proof}[Proof of Theorem \ref{normalisomorphism}]
By Proposition \ref{smoothextension}, it suffices to show that every metric current $T\in\normallocvanishing{k}(\BBG_R)$ satisfies the linearity and locality axioms of Definition \ref{currentdef} for smooth forms, as well as the local finiteness of $||T||$ and $||\partial T||$ when defined using smooth forms, and \textit{with respect to the metric $\cdist$}.

Linearity and locality follow from the fact that $T$ is a current.  To prove the local finiteness of $||T||$ and $||\partial T||$, we first let $\{v_1,\dotsc,v_m\}$ be a basis for the vertical subspace $\mathfrak v\subset \mathfrak g$.  Since by definition, $\pi|_V=0$, we have $\langle \dpansu \pi^i, v_j\rangle=0$ for $i=1,\dotsc,n$ and $j=1,\dotsc,m$, so that the dual basis to $\{u_1,\dotsc,u_n,v_1,\dotsc,v_m\}$ is $\{\dpansu \pi^1,\dotsc,\dpansu \pi^n, \theta^1,\dotsc,\theta^m\}$, where $\theta^1,\dotsc,\theta^m$ are vertical.  Thus for every smooth function $g\in\smooth(\BBG)$, we have
\[ \driem g=\sum_{i=1}^n X^{u_i}(g)\,\dpansu \pi^i + \sum_{j=1}^m X^{v_j}(g)\theta^j \text{.}\]
Moreover, by Rademacher's Theorem, $\driem$ is the Cheeger differential for the Riemannian metric, so Theorem \ref{representation} implies that 
\[ T\restrict{\driem g}=T\restrict{\sum_{i=1}^n X^{u_i}(g)\,\dpansu \pi^i + \sum_{j=1}^m X^{v_j}(g)\theta^j} \text{,}\]
and since $T$ vanishes on vertical forms, we further have that 
\begin{equation}
\label{rcurrentrestriction}
T\restrict{\driem g}=T\restrict{\sum_{i=1}^n X^{u_i}(g)\,d\pi^i} \text{.}
\end{equation}
As in the proof of Lemma \eqref{prefinitemass}, we let $\omega \in\smoothformsc{k}(U)$ with $||\omega||\leq 1$ (where comass is defined using smooth forms, and the metric $\cdist$), and let $U\subset \BBG$, with $\overline{U}$ compact. We may then write $\omega = \sum_{s \in S} f_s\,dg^1_s\wedge\dotsb\wedge dg^k_s$, with $f_s\in\smoothc(U)$, $\sum_{s\in S}|f_s|\leq (1+\epsilon)$ for some $\epsilon>0$,  and $g^i_s\in\lip_1(U)\cap \smooth(U)$ for each $i$ and $s$.  Note that under these assumptions, we have $|X^{u_i}(g^j_s)|\leq 1$.  Moreover, the projection $\pi$ was defined so that the Lipschitz constant with respect to the \textit{Riemannian} metric of each function $\pi^i$ is $1$.  Thus for each $s\in S$, we have
\[\left|\left|f_s \left(\sum_{i=1}^n X^{u_i}(g^1_s)\,d\pi^i\right)\wedge \dotsb\wedge \left(\sum_{i=1}^n X^{u_i}(g^k_s)\,d\pi^i\right)\right|\right|_r \leq \frac{n!}{k!} |f_s|\text{,}\]
where $||\cdot||_r$ denotes the comass of a metric form with respect to the Riemannian metric. It follows that $||\omega||_r\leq (1+\epsilon)\frac{n!}{k!}$. This holds for all $\epsilon>0$, so we have $||\omega||_r\leq \frac{n!}{k!}$.  Thus $|T(\omega)|\leq \frac{n!}{k!}||T||(U)<\infty$ for each open set $U$, by the local finiteness of $||T||$.  The argument for $\partial T$ is identical, and so by Proposition \ref{smoothextension}, the proof is complete.
\end{proof}

%% file: Rectifiability.tex
\section{Rectifiability}
\label{rectifiability}
We interpret our results in the context of rectifiable sets in metric spaces.

\begin{definition}
\label{rectdef}
A metric space $X$ is called \textbf{$k$-rectifiable} if it is the union of countably many Lipschitz images of subsets of $\BBR^k$ and an $\mathcal H^k$-null set.  That is,
$$X = \left(\bigcup_i F_i(A_i) \right) \cup N$$
where each $A_i \subseteq \BBR^k$, $F_i\colon A_i\rightarrow X$ is Lipschitz, and $\mathcal H^k(N)=0$. If every $k$-rectifiable subset $S$ of a space $X$ is trivial (i.e. $\mathcal H^k(S)=0$),  $X$ is said to be \textbf{purely $k$-unrectifiable}.
\end{definition}

Ambrosio and Kirchheim studied rectifiable sets in metric spaces in \cite{AmbrosioKirchheimRect}, continuing earlier work by Kirchheim \cite{Kirchheim}.  With the help of an area formula and a metric differentiation theorem developed in \cite{Kirchheim}, they proved that one can take the maps $F_i$ in Definition \ref{rectdef} to be bi-Lipschitz.  This immediately implies that a nontrivial $k$-rectifiable set must admit nonzero metric $k$-currents, as one can simply push forward a Euclidean current from one of the sets $A_i$.

We now examine some consequences of our results in terms of rectifiability.  First, Corollary \ref{fewcurrents} has immediate implications for the dimension of a rectifiable subset of a space admitting a differentiable structure.
\begin{corollary}
\label{norect}
Let $X = (X,d,\mu)$ be a proper, doubling, metric measure space admitting a differentiable structure of dimension $n$.  Then there is subset $N\subset X$, with $\mu(N)=0$, such that $X\backslash N$ is purely $k$-unrectifiable for any $k>n$.
\end{corollary}

The area formula and metric differential are also used in \cite{AmbrosioKirchheimRect} to prove the following theorem.  Though stated there for $n=1$, the proof given in \cite{AmbrosioKirchheim} extends to the general case.
\begin{theorem}[{\cite[Theorem 7.2]{AmbrosioKirchheimRect}}]
\label{norectheis}
The Heisenberg group $\heis{n}=(\heis{n},\cdist)$ is purely $k$-unrectifiable for $k>n$.
\end{theorem}

In light of the fact that one can use bi-Lipschitz maps in the definition of rectifiability, it is clear that Theorem \ref{norectheis} can also be viewed as a consequence of Corollary \ref{nohncurrents}.  On the one hand, this argument for unrectifiability is not much different from the one in \cite{AmbrosioKirchheimRect}, in that it uses the same ingredients, namely, Pansu's differentiation theorem and the area formula.  On the other hand,  the method of proof by way of currents uses the area formula only implicitly, and solely for the purpose of using bi-Lipschitz maps in Definition \ref{rectdef}.  Moreover,  this argument relies on differentiation of maps from $\heis{1}$ into Euclidean spaces, rather than vice-versa.  Thus no analysis of the metric differential of any map \textit{into} $\heis{1}$ is required. Instead, one computes the Cheeger differential of a map from $\heis{1}$ into a Euclidean space. 

Magnani \cite{Magnani} generalized the results of \cite{AmbrosioKirchheimRect} to arbitrary Carnot groups. 
\begin{theorem}[{\cite[Theorem 1.1]{Magnani}}]
\label{carnotrect}
A Carnot group $\BBG$ is purely $k$-unrectifiable if and only if every horizontal Abelian subalgebra of its Lie algebra $\mathfrak g$ has rank less than $k$.
\end{theorem}

We can interpret this result in the context of currents as well.  Indeed, suppose we pick linearly independent horizontal vectors $u_1,\dotsc,u_k \in H$, and let $\tilde{u}=u_1\wedge\dotsb\wedge u_k$. 

By the boundary computation \eqref{boundarycomputation}, the boundary $\partial T_{\tilde{u}}$ of the simple $k$-current $T_{\tilde{u}}$ vanishes if and only if $[u_i,u_j]=0$ for all $i$ and $j$.  This in turn is true if and only if the Lie subalgebra generated by the vectors $u_1,\dotsc,u_k$ is Abelian (or, equivalently, is horizontal).  Combining this with Proposition \ref{currentsk}, we obtain the following corollary to Magnani's Theorem.

\begin{corollary}
\label{carnotsimple}
A Carnot group $\BBG$ has a nontrivial $k$-rectifiable subset if and only if it has a nonzero, invariant, ``simple'' $k$-current $T_{\tilde{u}}=T_{u_1 \wedge\dotsb\wedge u_k}$.
\end{corollary}

We are unaware if Corollary \ref{carnotsimple} can be deduced independently of Theorem \ref{carnotrect}.  In particular, we do not know whether either implication is true in a general metric group with a differentiable structure.

\begin{remark}
It is not true that the absence of $k$-rectifiable sets in $\BBG$ implies the nonexistence of arbitrary (i.e., non-simple) $k$-currents.  To construct an explicit counterexample, let $\mathfrak g$ have the stratification $$\mathfrak g = \spa(u_1,u_2,u_3,u_4)\oplus \spa(v_1,v_2,v_3,v_4,v_5)$$ satisfying the relations $[u_1,u_2]=[u_3,u_4] = v_1$, $[u_1,u_3]=v_2$, $[u_1,u_4]=v_3$, $[u_2,u_3]=v_4$, and $[u_2,u_4]=v_5$.  It is easily verified that any two linearly independent horizontal vectors do not commute, and so by Corollary \ref{carnotsimple} and Theorem \ref{carnotrect}, respectively, $\BBG$ admits no nonzero simple $2$-currents, nor any nontrivial $2$-rectifiable sets.  On the other hand,  $T_{u_1\wedge u_2 - u_3 \wedge u_4}$ is a $2$-current, and is in fact a cycle, again by equation \eqref{boundarycomputation}.  Thus there are purely $k$-unrectifiable spaces which still admit normal $k$-currents, for $k\geq 2$.  In this sense, the theory of metric currents is at least somewhat more general than the theory of rectifiable sets.  This contrasts starkly with the Euclidean case, where every normal metric current can be identified with a normal current in the sense of Federer and Fleming \cite[Theorem 5.5]{Lang}, and where the latter can be approximated in Whitney's flat norm \cite{Whitney} (and hence weakly) by polyhedral chains, which are of course rectifiable.
\end{remark}